\documentclass[12pt]{article}
\usepackage{mathrsfs}
\usepackage{stmaryrd}
\usepackage{amsfonts}
\usepackage{}
\usepackage{epsfig}
\usepackage{tikz}
\usepackage{graphicx}
\usepackage{enumerate}
\usepackage{amsthm}
\usepackage{amssymb}
\usepackage{caption}
\usepackage{subcaption}
\usepackage{color}
\usepackage{amsmath,mathrsfs,amsfonts,verbatim,enumitem,leftidx}
\usepackage{indentfirst}
\usepackage[colorlinks,linkcolor=red,anchorcolor=blue,citecolor=blue]{hyperref}
\usepackage{cite}

\newtheorem{thm}{Theorem}[section]
\newtheorem{cro}[thm]{Corollary}
\newtheorem{defn}{Definition}[section]
\newtheorem{prop}{Proposition}[section]

\newtheorem{lem}[thm]{Lemma}

\numberwithin{equation}{section}

\newtheorem{exa}[thm]{Example}

\newcommand*{\dif}{\mathop{}\!\mathrm{d}}

\begin{document}
\title{
On the  metric mean dimensions of saturated sets
\footnotetext {* Corresponding author}
}
\author{Yong Ji$^{1}$, Junye Li$^{1}$ and Rui Yang$^{*2,3}$ \\ 
\small 1.  Department of Mathematics, Ningbo University, Ningbo 315211, Zhejiang, P.R. China\\ 
\small 2.College of Mathematics and Statistics, Chongqing University, Chongqing 401331, P.R.China\\
\small 3. Key Laboratory of Nonlinear Analysis and its Applications (Chongqing University), \\
\small Ministry of Education\\
\small e-mail: imjiyong@126.com, lijunyegz@163.com, zkyangrui2015@163.com}

\date{}
\maketitle

\begin{center}
\begin{minipage}{120mm}
{\small {\bf Abstract.}  From a geometric perspective, we employ metric mean dimension to investigate the set of generic points of invariant measures and saturated sets in infinite entropy systems. For systems with the specification property, we establish certain variational principles for the Bowen and packing metric mean dimensions of saturated sets in terms of Kolmogorov-Sinai $\epsilon$-entropy, and prove that the upper capacity metric mean dimension of saturated sets has full metric mean dimension. Consequently, the Bowen and packing metric mean dimensions of the set of generic points of invariant measures coincide with the mean Rényi information dimension, and the upper capacity metric mean dimension of the set of generic points of invariant measures also has full metric mean dimension.
	
~~~~~As applications, for systems with the specification property, we present the qualitative characterization of the metric mean dimensions of level sets, the set of mean Li-Yorke pairs in infinite-entropy systems, and the set of generic points of invariant measures in full shifts over compact metric spaces.

}
\end{minipage}

\end{center}

\vskip0.5cm
\tableofcontents

\section{Introduction}

Let $(X,d,f)$ be a topological dynamical system (TDS for short), where $(X,d)$ is a compact metric space, and $f:X\to X$ is  a homeomorphism.   Let $M(X)$, $M_f(X)$, $M_f^e(X)$ denote the sets of Borel probability measures on $X$ endowed with the weak$^\ast$-topology, $f$-invariant, and $f$-ergodic  Borel probability measures on $X$,  respectively. 

Quantitative analysis of dynamical systems aims to use  entropy-like quantities to understand the dynamics of the abstract dynamical systems.  Among these  entropy-like quantities,  measure-theoretic entropy and topological entropy are two  topological invariants to characterize the topological complexity of dynamical systems, which  are  related by the classical variational principle:
$$h_{top}(f)=\sup_{\mu \in M_f(X)}{h_{\mu}(f)},$$
where $h_{top}(f)$  denotes the  topological entropy of $X$, and  $h_{\mu}(f)$ is the measure-theoretic entropy of  $\mu$. To extend the concept of topological entropy to any non-empty subsets, resembling the definition of Hausdorff dimension, Bowen  \cite{Bowen1973}  in his pioneering work  introduced Bowen topological entropy on subsets. It turns out that the Bowen topological entropy  provides a powerful tool for studying the ``size"  of fractal-like sets and has directly contributed to the   development of multifractal analysis in  dynamical systems.  The multifractal analysis, with its origin dating back to Besicovitch, takes into account the decomposition   of the whole phase space into smaller subsets consisting of  points with the similar dynamical behaviors. And then, from a geometrical  and topological perspectives, we use the  fractal dimension and entropy-like quantity  to  describe the “size” of these subsets.  Such a procedure allows us  to obtain  more information  about the dynamics  and even recovers the partial dynamics. This  phenomenon is also known as \emph{multifractal rigidity}.  

 Typically, infinite entropy systems  exhibit rather complicated topological complexity. As a result, the concept of topological entropy fails to offer  more information about the dynamics of the system. In 2000,    Lindenstrauss and Weiss \cite{Lindenstrauss2000} introduced  the metric mean dimension, which is  a dynamical analogue of box dimension. The  measure-theoretic counterpart of metric mean dimension, called  upper and lower mean R\'enyi information dimensions, is due to Gutman and \'Spiewak\cite{Gutman2021}.  
The analogous variational principles for  metric mean dimension   can be found in the references \cite {Lindenstrauss2018, Gutman2021,ycz24, Shi2022}.  Given $x\in X$, denote by $$\mathcal{E}_{n}(x):=\frac{1}{n}\sum_{i=0}^{n-1}\delta_{f^ix}$$
the $n$-th empirical measure of  $x$, where $\delta_{x}$ is the Dirac  measure at $x$. It is well-known that,  in the   weak$^\ast$-topology,  the limit-point set  $V_{f}(x)$ of the sequence $\{\mathcal{E}_{n}(x)\}_{n\geq 1}$   is  a non-empty compact connected subset of $M_f(X)$ .
In particular, if  $V_f(x)=\{\mu\}$ for some  $f$-invariant measure $\mu$, then  $x$ is  called a generic point of $\mu$. Denote by $G_{\mu}$ the set of all generic points of $\mu$. The Birkhoff  ergodic theorem shows that  $\mu(G_{\mu})=1$ if and only if $\mu$ is  ergodic.


We aim to    study the   multifractal analysis of the set of generic points in infinite entropy systems. Let us  mention some   results  about the  topological entropy of  the set of generic points.  For every  ergodic measure $\mu$,  Bowen \cite{Bowen1973} proved that the Bowen topological entropy of $G_{\mu}$ equals  $h_{\mu}(f)$, i.e.,
$$h_{top}^B(f,G_{\mu})=h_{\mu}(f).$$
For the non-ergodic measure $\mu$,  one has $\mu(G_{\mu})=0$, and there  exist  the examples showing that  $G_{\mu}=\emptyset$ but $h_{\mu}(f)>0$. Therefore, Bowen's result fails for  non-ergodic measures. In 2007,  Pfister and Sullivan \cite {Pfister2007} proved that Bowen's  equality still holds for non-ergodic measures  within a more general framework.  More generally, given  a non-empty  closed  connected subset $K\subset M_f(X)$,  the saturated set $G_K$ associated with $K$  is defined by  $$G_K=\{ x\in X: V_f(x)=K \}.$$  By   introducing the $g$-almost product property  (e.g. $\beta$-shifts),  which  is strictly weaker than  the specification property,  and the  uniform separation property (e.g. expansive and more  generally  asymptotically $h$-expansive systems),  they  established a  variational principle for $G_K$: 
$$h_{top}^{B}(T,G_K)=\inf_{\mu \in K}h_{\mu}(f).$$
Furthermore, under only the assumption of the  $g$-almost  product property, they \cite{Pfister2007} showed  that  $h_{top}^B(f,G_{\mu})=h_{\mu}(f)$ for all  $\mu \in  M_f(X)$. For systems with the specification property  and  positive expansive property, Zhou, Chen and Cheng \cite{Zhou2012}  also obtained a (dual) variational principle for packing topological entropy of $G_K$:
$$h_{top}^{P}(f,G_K)=\sup_{\mu \in K}h_{\mu}(f).$$ 
Later, for systems with the  $g$-almost product property,  Hou,  Tian and Zhang \cite{Hou2023} proved that  the upper capacity  topological entropy of the saturated sets with more complicated dynamics  has full topological entropy, i.e.,
\begin{equation*}
	h_{top}^{UC}(f,G_{K}^{C})=h_{top}(f),
\end{equation*}
where  $G_{K}^{C} := G_{K} \cap \{x \in X : C_{f}(X) \subset \omega_{f}(x)\},$  $\omega_{f}(x):=\bigcap_{n=0}^{\infty} \overline{\bigcup_{k=n}^{\infty}\{f^{k}x\}}$ is the $\omega$-limit set of $x$, and  $C_{f}(X)$ is  the measure center given by $C_{f}(X):= \overline{\cup_{\mu\in M_{f}(X)}{\rm supp}{(\mu)}}$, ${\rm supp}(\mu)$ is the supported set of  $\mu$. Then, in the context of  the infinite entropy systems, the aforementioned work suggests the following questions for  saturated sets:\\
\emph{Questions:}
\begin{itemize}
	\item [(1)] For ergodic measures, does the metric mean dimension of the set of generic points  coincide with  the  mean R\'enyi information dimension?
	\item [(2)] For non-ergodic measures,  under the assumption of specification-like property, are there  certain variational principles for metric mean dimensions of saturated sets? If not, can  metric mean dimensions of saturated sets have the full metric mean dimension?
\end{itemize}

 For (1), the authors in \cite{ycz24} proved  that the packing metric mean dimension of the set of generic points  equals  the  mean R\'enyi information dimension. However, this equality  holds for Bowen metric mean dimension under the assumption of  an analogous Brin-Katok formula  within the context of infinite entropy systems (\cite[\S 5, Question,(1)]{ycz24}); for (2), since the  uniform separation property implies the system has finite topological entropy,  extending the  work of \cite{Pfister2007,Zhou2012, Hou2023} to the framework of infinite  entropy systems encounters some nontrivial challenges.  
 Actually,   since the asymptotic behavior of orbits ${\rm orb}_f(x)=\{ f^n(x):n\geq 0\}$ is a fundamental topic in ergodic theory,  the subsequent statements  of main results  involve   more general (mixed)  saturated sets. For simplicity,  we write 
 \begin{equation*}
 	G_K^{\omega,U}:=G_K\cap U\cap \{ x\in X: \omega_f(x)=X \},
 \end{equation*}
 where $U\subset X$ is a non-empty open subset, $K\subset M_f(X)$ is a non-empty closed connected (or convex) subset.  If  $G_K^{\omega,U}\not=\emptyset$, then there exists a transitive point in  the sufficiently ``small" open set $U$  such that $V_f(x)=K$. For systems with the specification property, we provide a complete answer to question (2).
 
  Our main results clarify that the saturated sets exhibit  different dynamical behaviors  in terms of different types of metric mean dimensions. There are certain variational principles for the Bowen and packing metric mean dimensions of saturated sets, as presented in Theorems \ref {thm 1.1}, \ref{thm 1.2} and \ref {thm 1.3}. 


Let  $\overline{\rm mdim}_M^{B}(f,Z,d)$, $\overline{\rm mdim}_M^{P}(f,Z,d)$ and $\overline{\rm mdim}_M^{UC}(f,Z,d)$ denote    the Bowen/packing/upper capacity  metric mean dimension of $Z$, respectively.

\begin{thm}\label{thm 1.1}
Let $(X,d,f)$ be a TDS with  the specification property. If $K\subseteq M_{f}(X)$ is a non-empty compact convex set, and $U\subset X$ is a  non-empty open set, then 
\begin{align*}
\overline{\rm mdim}_M^{P}\left(f,G_K,d\right)=&\overline{\rm mdim}_M^{P}(f,G_K^{\omega,U},d)
=\limsup_{\epsilon\to0}\frac{1}{|\log\epsilon|}
 \sup_{\mu\in K}
 \inf_{  {\rm diam}(\xi)<\epsilon }h_\mu(f,\xi).
\end{align*}
\end{thm}
\begin{thm}\label{thm 1.2}
Let $(X,d,f)$ be a TDS with  the specification property. If $K\subseteq M_{f}(X)$ is a non-empty compact connected set, and $U\subset X$ is a  non-empty open set, then 
\begin{align*}
\overline{\rm mdim}_M^{B}\left(f,G_K,d\right)=&\overline{\rm mdim}_M^{B}(f,G_K^{\omega,U},d)
=\limsup_{\epsilon\to0}\frac{1}{|\log\epsilon|}
\inf_{\mu\in K}
\inf_{  {\rm diam}(\xi)<\epsilon }h_\mu(f,\xi).
\end{align*}
\end{thm}

\begin{thm}\label{thm 1.3}
Let $(X,d,f)$ be a TDS with  the specification property. If $K\subseteq M_{f}(X)$ is a non-empty compact  convex set, then
\begin{align*}
\underline{\rm mdim}_{M}^{B}(f,G_{K},d)=\inf_{\mu \in K}\underline{\rm mrid}_{\mu}^B(f,K,d),\\
\overline{\rm mdim}_{M}^{P}(f,G_{K},d)=\sup_{\mu \in K}\overline{\rm mrid}_{\mu}^P(f,K,d),
\end{align*}
where  $\underline{\rm mrid}_{\mu}^B(f,K,d)$ and $\overline{\rm mrid}_{\mu}^P(f,K,d)$ denotes the Bowen and packing  R\'enyi information dimensions of $\mu$, respectively.
\end{thm}

We remark that  any  convex set in a topological vector space is automatically connected. Although  the $\mu$-measure of $G_{\mu}$ is zero for every non-ergodic measure  $\mu$, its topological complexity  can be sufficiently large,  that is, the upper capacity metric mean dimension of saturated sets  carries the full metric mean dimension, as stated in Theorem \ref {thm 1.4}.

\begin{thm}\label{thm 1.4}
Let $(X,d,f)$ be a TDS with  the specification property. Let $K\subseteq M_{f}(X)$ be a non-empty compact connected set, and $U\subset X$ be a non-empty open set. Then
\begin{align*}
\overline{\rm mdim}_M^{UC}\left(f,G_K,d\right)
=\overline{\rm mdim}_M^{UC}(f,G_K^{\omega,U},d)
=\overline{\rm mdim}_M(f,X,d),
\end{align*}
where $\overline{\rm mdim}_M^{UC}(f,Z,d)$  is the  upper  capacity metric mean dimension of $Z$. 
\end{thm}

As  an immediate consequence of Theorems \ref{thm 1.1}, \ref{thm 1.2} and \ref{thm 1.4}, the most interesting case  occurs when 
$K\subset M_f(X)$ consists of exactly a single point. The following Corollary \ref{cor 1.5}  extends the  result  \cite[Theorem 1.1]{ycz24} to all non-ergodic measures  for systems with  the specification property.   

\begin{cro}\label{cor 1.5}
Let $(X,d,f)$ be a TDS with  the specification property.
Then  for  every  $\mu\in M_f(X)$,
\begin{align*}
&\overline{\rm mdim}^{B}_M(f,G_\mu,d)= 
\overline{\rm mdim}^{P}_M(f,G_\mu,d)=
\limsup_{\epsilon\to0}\frac{1}{|\log\epsilon|}\inf_{{\rm diam}(\xi)<\epsilon}\limits h_\mu(f,\xi),\\
&\overline{\rm mdim}_M^{UC}(f,G_\mu,d)=\overline{\rm mdim}_M(f,X,d).
\end{align*}
\end{cro}

Except  Theorem \ref{thm 1.3}, all theorems mentioned above are valid for the corresponding lower metric mean dimensions by considering changing the $\limsup_{\epsilon\to 0}$ into $\liminf_{\epsilon\to 0}$. Besides, it becomes unclear whether the  aforementioned results for weaker specification-like property.

The paper is organized as follows. In section \ref{sec 2}, we recall the precise definitions of metric mean dimensions and  the associated properties, and two types of measure-theoretic $\epsilon$-entropies, and the specification property of TDSs. In section \ref{sec 3}, we  prove the main results. In section \ref{sec 4}, we consider some applications of  the main results, and apply it  to investigate the level sets of the Birkhoff average of a continuous potential, Li-Yorke chaos in finite entropy systems and  the metric mean dimensions of generic points of the full shift over  compact metric spaces.

\section{Preliminaries}\label{sec 2}

In this section, we recall the concepts of metric mean dimensions in both topological and measure-theoretic situations, and state the  precise definition  of the specification  property in  dynamical systems.

\subsection{Three types of metric mean dimensions}

In this subsection, we review the definitions of upper capacity  metric mean dimension of subsets\cite{Lindenstrauss2000},  Bowen  metric mean dimension\cite{lp21, Wang2021} and packing metric mean dimension \cite{ycz22} of subsets.

Let $(X,d,f)$ be a TDS.
For any natural number $n\in \mathbb{N}$, we define  the $n$-th Bowen metric $d_{n}$ on $X$ as
\begin{equation*}
d_{n}(x,y):=\max_{0\le i\le n-1}d(f^{i}x,f^{i}y),
\end{equation*}
where  $x,y \in X$. Then the Bowen open ball and closed ball centered at $x$ with radius $\epsilon$ in the metric $d_n$ are given by
\begin{align*}
B_{n}(x,\epsilon):&=\{y\in X:d_{n}(x,y)<\epsilon\},\\
\overline{B} _{n}(x,\epsilon):&=\{y\in X:d_{n}(x,y)\le \epsilon\}
\end{align*}
respectively.
 
Let $Z$ be a non-empty subset of $X$. Fix $\epsilon>0$. A set $E\subset Z$ is said to be an \emph{$(n,\epsilon)$-separated set of $Z$} if  any distinct $x,y\in E$  implies $d_{n}(x,y)>\epsilon$.  Denote by $s_{n}(Z,\epsilon)$  the largest cardinality of $(n,\epsilon)$-separated sets of $Z$. The  \emph{$\epsilon$-upper capacity topological entropy of $Z$} is given by
\begin{equation*}
h_{top}^{UC}(f,Z,\epsilon)=\limsup_{n\to\infty}\frac{1}{n}\log s_{n}(Z,\epsilon).
\end{equation*}
 Recall that  the  \emph{upper capacity  topological entropy of $Z$} is  the limit of  $\epsilon$-upper capacity topological entropy of $Z$ as $\epsilon$ goes to $0$, i.e., $$h_{top}^{UC}(f,Z)=\lim_{\epsilon\to0}h_{top}^{UC}(f,Z,\epsilon)=\sup_{\epsilon >0}h_{top}^{UC}(f,Z,\epsilon).$$
To get more information  about the dynamics  of  infinite entropy systems, analogous to the definition of box dimension in fractal  geometry,  we formulate a refined entropy-like quantity   using  the $\epsilon$-upper capacity topological entropy.
\begin{defn} 
The upper capacity metric mean dimension of $Z$   is defined by
\begin{equation*}
\overline{\rm mdim}_{M}^{UC}(f,Z,d):=\limsup_{\epsilon\to 0}\frac{h_{top}^{UC}(f,Z,\epsilon)}{|\log\epsilon|}.
\end{equation*}
\end{defn}

If $Z=X$,  the upper capacity metric mean dimension of $X$ is reduced to the metric mean dimension introduced by Lindenstrauss-Weiss \cite{Lindenstrauss2000}; in this case, we use the notation $h_{top}(f,X,\epsilon)$ instead of $h_{top}^{UC}(f,X,\epsilon)$,  and $\overline{\rm mdim}_{M}(f,X,d)$ instead of $\overline{\rm mdim}_{M}^{UC}(f,X,d)$, respectively.   Such a formulation extends the concept of  metric mean dimension to any subsets of the phase space, which is not necessarily compact or  $f$-invariant. Notice that  the (upper capacity) metric mean dimension  is defined by measuring how fast the  $\epsilon$-upper capacity topological entropy  diverges to $\infty$  as $\epsilon \to 0$. That is, for sufficiently small $\epsilon >0$, we may think of  
 $$h_{top}^{UC}(f,Z,\epsilon)\approx\overline{\rm mdim}_{M}^{UC}(f,Z,d)\cdot|\log \epsilon|.$$
 It is easy  to see that any system with finite  topological entropy has zero  metric mean dimension. This shows  that the metric mean dimension is a suitable ``measure" for capturing the dynamics of the  infinite entropy systems.

 The  Bowen topological entropy\cite{Bowen1973} and packing topological entropy\cite{Feng2012} in dimension theory,  which are defined by Carath\'eodory-Pesin structures, are the analogues of  the  Hausdorff dimension and packing dimension  in fractal geometry. The corresponding dimensional characterizations for metric mean dimension, which have been introduced in \cite{lp21,ycz22,Wang2021} through  Carath\'eodory-Pesin structures, are the following Bowen and packing metric mean dimensions.

 For $Z\subset X$, $s>0$, $N\in\mathbb{N}$ and $\epsilon>0$, we put
 \begin{equation*}
 M(Z,s,N,\epsilon)=\inf\left\{
 \sum_{i\in I}\exp(-n_is)
 \right\},
 \end{equation*}
 where the infimum is taken over all finite or countable families $\{B_{n_i}(x_i,\epsilon)\}_{i\in I}$ such that $Z\subset \cup_{i \in I} B_{n_i}(x_i,\epsilon)$ with $x_i \in X, n_i\geq N$ for all $i \in I$.
 
 The quantity $M(Z,s,N,\epsilon)$ is non-decreasing as $N$ increases. So  the limit
 \begin{equation*}
 M(Z,s,\epsilon)=\lim_{N\to\infty}M(Z,s,N,\epsilon)
 \end{equation*}
 exists.
 There is a  critical value  of  the parameter $s$, which we call $\epsilon$-Bowen topological entropy,  such that   $M(Z,s,\epsilon)$ jumps  from $\infty$ to $0$, i.e.,
 \begin{align*}
 h_{top}^B(f,Z,\epsilon)&=\inf\{s>0:M(Z,s,\epsilon)=0 \}\\
 &=\sup\{s>0:M(Z,s,\epsilon)=\infty \}.
 \end{align*} 

\begin{defn}
The Bowen upper metric mean dimension of $Z$ is defined by 
\begin{equation*}
\overline{\rm mdim}^B_M(f,Z,d)=\limsup_{\epsilon\to 0}\frac{h_{top}^B(f,Z,\epsilon)}{|\log\epsilon|}.
\end{equation*}
\end{defn}

For $Z\subseteq X$, $s>0$, $N\in\mathbb{N}$ and $\epsilon>0$, we put
\begin{equation*}
P(Z,s,N,\epsilon)=\sup\left\{\sum_{i\in I}\exp(-sn_{i})\right\},
\end{equation*}
where the supremum is taken over all finite or countable pairwise disjoint families $\{\overline{B} _{n_{i}}(x_{i},\epsilon)\}_{i\in I}$ with $x_{i}\in Z $, $n_{i}\ge N$ for all $i\in I$. 

Since the quantity $P(Z,s,N,\epsilon)$ is non-increasing  as $N$ increases, and thus the following limit exists:
\begin{equation*}
P(Z,s,\epsilon)=\lim_{N\to\infty}P(Z,s,N,\epsilon).
\end{equation*}
Observe that the quantity $ P(Z,s,\epsilon)$   does not satisfy the countable additive property with respect to $Z$. So the modifications are needed to obtain this property:
\begin{equation*}
\mathcal{P}(Z,s,\epsilon):=\inf\left\{\sum_{i=1}^{\infty}P(Z_{i},s,\epsilon):\bigcup_{i=1}^{\infty}Z_{i}\supseteq Z\right\}.
\end{equation*}
There is a  critical value  of  the parameter $s$, which we call $\epsilon$-packing topological entropy,  such that   $\mathcal{P}(Z,s,\epsilon)$ jumps  from $\infty$ to $0$. That is,  the  critical value is given by  
\begin{align*}
h_{top}^P(f,Z,\epsilon)&=\inf\{s>0:\mathcal{P}(Z,s,\epsilon)=0 \}\\
&=\sup\{s>0:\mathcal{P}(Z,s,\epsilon)=\infty \}.
\end{align*} 

\begin{defn}
The packing  upper metric mean dimension of $Z$ is defined by 
\begin{equation*}
\overline{\rm mdim}_{M}^{P}(f,Z,d)=\limsup_{\epsilon\to0}\frac{h_{top}^{P}(f,Z,\epsilon)}{|\log\epsilon|}.
\end{equation*}
\end{defn}

One can  replace $\limsup_{\epsilon\to 0}$ by $\liminf_{\epsilon\to 0}$ in above definitions to  obtain  the   corresponding  three types of lower  metric mean dimensions on subsets, which are denoted by  $\underline{\rm mdim}_{M}^{S}(f,Z,d)$ with $S\in \{UC,B,P\}$.  If $\underline{\rm mdim}_{M}^{S}(f,Z,d)=\overline{\rm mdim}_{M}^{S}(f,Z,d)$, we call the common value ${\rm mdim}_{M}^{S}(f,Z,d)$ the (capacity/Bowen/packing) metric mean dimension of $Z$.

 We summarize some fundamental properties of metric mean dimensions derived in  \cite[Proposition 3.4]{ycz22} and  \cite[Proposition 2.8]{ycz24c}.

\begin{prop}\label{prop 2.1}
$(1)$   If $Z_1\subset Z_2 \subset X$, then for   each $S\in \{UC,B,P\}$,
\begin{align*}
\overline{\rm mdim}_M^{S}(f,Z_1,d)\leq \overline{\rm mdim}_M^{S}(f,Z_2,d).
\end{align*}

$(2)$ If $Z=\bigcup_{i=1}^{\infty}Z_{i}$,  then for every $\epsilon >0$,
\begin{align*}
 h_{top}^{B}(f,Z,\epsilon)&=\sup_{i\ge1}h_{top}^{B}(f,Z_{i},\epsilon),\\
 h_{top}^{P}(f,Z,\epsilon)&=\sup_{i\ge1}h_{top}^{P}(f,Z_{i},\epsilon),
\end{align*}
and  $ h_{top}^{UC}(f,Z,\epsilon)=\max_{1\leq i\leq N}h_{top}^{UC}(f,Z_{i},\epsilon)$ if  $Z=\bigcup_{i=1}^{N}Z_{i}$.

$(3)$ If  $Z=\bigcup_{i=1}^{N}Z_{i}$, then  for   each $S\in \{UC,B,P\}$,
\begin{align*}
\overline{\rm mdim}_M^{S}(f,Z,d)&= \max_{1\leq i\leq N} \overline{\rm mdim}_M^{S}(f,Z_i,d).
\end{align*}

$(3)$ If $Z$ is a non-empty subset of $X$,  then 
for every  $0<\epsilon<1$, we have
$$h_{top}^{B}(f,Z,3\epsilon)\leq h_{top}^{P}(f,Z,\epsilon)\leq h_{top}^{UC}(f,Z,\epsilon).$$ Consequently,
$\overline{\rm mdim}_{M}^{B}(f,Z,d)\leq \overline{\rm mdim}_{M}^{P}(f,Z,d)\leq \overline{\rm mdim}_{M}^{UC}(f,Z,d)$. 

Furthermore, if $Z$ is $f$-invariant and compact, then the  three types of metric mean dimensions coincide.
\end{prop}

Given  a certain subset of phase space, in general it is not easy to  calculate its precise Bowen and packing metric mean dimensions. The  \emph{Entropy Distribution Principles} allow us to  get a lower bound for  $\epsilon$-Bowen and $\epsilon$-packing topological entropies of  certain fractal-like sets, and hence yields the lower bounds of  metric mean dimension of the fractal-like sets.

 The Entropy Distribution Principles  of  $\epsilon$-Bowen and packing topological entropies  are stated  as  follows:

\begin{lem}{\rm \cite[Lemma 13]{Backes2023}}\label{bowdis}
Let $(X,d,f)$  be a TDS and $\epsilon>0$. Suppose that $Z$  is a  Borel set of $X$ and $0<s<\infty$. If there exist   $\mu \in M(X)$ and  a constant $C>0$  such that $\mu(Z)>0$ and $\mu(B_n(x,\epsilon) )\leq Ce^{-ns}$ for every each Bowen ball  $B_n(x,\epsilon)\cap Z\neq\emptyset$, 
then $$h_{top}^B(f,Z,\epsilon)\geq s.$$
\end{lem}

\begin{lem}\label{lem5}
Let $(X,d,f)$  be a TDS and $\epsilon>0$. Suppose that $Z$  is a  Borel set of $X$ and $0<s<\infty$. If there exist   $\mu \in M(X)$, a constant $C>0$ and  a strictly increasing sequence $\{n_i\}\subset\mathbb{N}$  such that $\mu(Z)>0$ and   $\mu(B_{n_{i}}(x,\epsilon))\leq Ce^{-n_{i}s}$ for each $i$ and for any $ x\in Z$,
then $$h^{P}_{top}(f, Z,\frac{\epsilon}{6})\ge s.$$
\end{lem}

\begin{proof}
  By definition of the $\epsilon$-packing topological entropy, it is enough to prove that for  every sufficiently small $\delta>0$, $$P\left(E,s-\delta,\frac{\epsilon}{6}\right)=\infty$$  for any Borel subset $E\subset Z$ with $\mu(E)>0$.
  
   For each $i$, consider the  cover  $\mathcal{A}_i=\{ B_{n_i}(x,\frac{\epsilon}{5}): x\in E \}$ of  $E$. Then, by  $5r$-covering lemma \cite[Theorem 2.1]{Mattila1995}, there exists a countable pairwise disjoint sub-collection $\{ B_{n_i}(x_j,\frac{\epsilon}{5})\}_{j\in J}$, where $J$ is a index set,  such that $E\subset  \bigcup_{j \in J}B_{n_i}(x_j,\epsilon).$ In particular,  the collection of closed Bowen balls $\{\overline{B}_{n_i}(x_j,\frac{\epsilon}{6}) \}_{j\in J}$ is pairwise disjoint.
Hence,  one has
\begin{align*}
P\left(E,s-\delta,n_i,\frac{\epsilon}{6}\right)
\geq
\sum_{j\in J}\exp[ -n_i(s-\delta) ]\geq\frac{e^{n_i\delta}}{C}
\sum_{j\in J}\mu( B_{n_i}(x_j,\epsilon) )\geq 
\frac{e^{n_i\delta}}{C}\mu(E).
\end{align*}
 Letting $i\to\infty$, we have $P(E,s-\delta,\frac{\epsilon}{6})=\infty$. This shows  that  $\mathcal{P}(Z,s-\delta,\frac{\epsilon}{6})=\infty$, and hence  $h^{P}_{top}(f, Z,\frac{\epsilon}{6})\ge s$ by letting $\delta \to 0$.
\end{proof}

\subsection{Measure-theoretic $\epsilon$-entropy of  invariant measures}
In this subsection, we recall  the definitions of Kolmogorov-Sinai $\epsilon$-entropy and
Katok  $\epsilon$-entropy of invariant measures.

Let $C(X)$ denote the  space of  real-valued continuous functions  on $X$,  equipped with the supremum norm $||\cdot||$.  By virtue of the Riesz representation theorem for compact metric spaces, we are able  to endow  $M(X)$ with the weak$^{*}$-topology, which is metrizable by the following metric:
\begin{equation}\label{metr}
D(\mu,\nu):=\sum_{k=1}^\infty 
\frac{|\int\phi_k\dif\mu-\int \phi_k\dif\nu|}{2^k},
\end{equation}
where $\{ \phi_k \}$ is a  countable family of  dense subsets of $C(X)$  taking values in $[0,1]$. Thus, the diameter of $M(X)$ in the sense of the metric $D$ is  always less than $1$, and  for fixed $\mu,\nu \in M(X)$,  both $D(\mu,\cdot)$ and $D(\cdot,\nu)$  are convex functions on  $M(X)$. Furthermore, it is well-known that $M(X)$ and $M_f(X)$ are compact  convex sets in  the weak$^{*}$-topology.

Given  $\mu\in M_{f}(X)$ and 
a  finite Borel partition $\xi=\left\{C_{1},\cdots,C_{k}\right\}$, the partition entropy of $\xi$  w.r.t. $\mu$ is defined by $H_{\mu}(\xi)=-\sum_{i=1}^{k}\mu(C_{i})\log (\mu(C_{i})).$
We define the  \emph{measure-theoretic entropy} of $\mu$ w.r.t. $\xi$ as
\begin{equation*}
h_{\mu}(f,\xi)=\lim_{n\to+\infty}\frac{1}{n}H_{\mu}(\xi^{n}),
\end{equation*}
Denote by  $\xi^n:=\vee_{i=0}^{n-1}f^{-i}\xi$ is  the $n$-th join of  the  preimage partitions $\xi, f^{-1}\xi,\ldots,f^{-(n-1)}\xi$ on $X$.
The \emph{measure-theoretic entropy of  $\mu$} is defined as $$h_{\mu}(f)=\sup_{\xi}h_{\mu}(f,\xi),$$
where the supremum ranges over  all finite Borel  partitions of $X$.

For topological dynamical systems, the measure-theoretic counterpart of metric mean dimension is the measure-theoretic metric mean dimension of invariant measures. It  measures the divergent rate of the limits of the measure-theoretic  $\epsilon$-entropies such as  Brin-Katok $\epsilon$-entropy, Katok $\epsilon$-entropy, Kolmogorov-Sinai $\epsilon$-entropy, and the others.  A systemic study about the measure-theoretic $\epsilon$-entropies  is stated in \cite{ycz24}. Here, we only recall two types of the measure-theoretic 
$\epsilon$-entropies.

Let $\mu \in M_f(X) $ and $\epsilon >0$. 
\begin{itemize}
\item [(a)]   The \emph{Kolmogorov-Sinai $\epsilon$-entropy} \cite{Gutman2021}  of $\mu $ is given by
\begin{equation*}
\inf_{{\rm diam}(\xi)<\epsilon}h_\mu(f,\xi),
\end{equation*}
where the infimum is taken over all finite Borel partitions of $X$ with diameter  less than $\epsilon$.
\item [(b)] For  $\delta\in(0,1)$ and $n\in\mathbb{N}$, we put
\begin{equation*}
N_\mu(n,\epsilon,\delta):=\min\{
\#E: E\subset X~\emph{and}~ \mu(\cup_{x\in E}B_n(x,\epsilon))>1-\delta
\}.
\end{equation*}
The \emph{Katok $\epsilon$-entropy}   of $\mu$  \cite{k80} is defined by 
\begin{equation*}
\underline{h}_\mu^K(f,\epsilon,\delta)=\liminf_{n\to\infty}\frac{1}{n}\log N_\mu(n,\epsilon,\delta).
\end{equation*}
\end{itemize}

  The  Katok $\epsilon$-entropy and Kolmogorov-Sinai $\epsilon$-entropy are related by the following inequalities:
\begin{lem}{\rm \cite[Lemma 4.1]{
Shi2022}}\label{ksmu}
Let $(X,d,f)$ be a TDS and $\mu\in M_f^e(X)$, If $\mathcal{U}$ is a finite open cover of $X$ with the diameter ${\rm diam}(\mathcal{U})\leq 4\epsilon$ and the Lebesgue number ${\rm Leb}(\mathcal{U})\geq\epsilon$, then for any $\delta\in(0,1)$, 
\begin{equation*}
\underline{h}_\mu^K(f,4\epsilon,\delta)
\leq\inf_{\xi\succ\mathcal{U}}h_\mu(f,\xi)
\leq 
\underline{h}_\mu^K(f,\epsilon,\delta),
\end{equation*}
where  the infimum $\xi$ ranges over all finite Borel partitions of $X$ refining than $\mathcal{U}$.

Consequently, for every $\epsilon>0$, one has $\underline{h}_\mu^K(f,16\epsilon,\delta)\leq \inf_{{\rm diam}(\xi)<4\epsilon}\limits h_\mu(f,\xi)\leq \underline{h}_\mu^K(f,\epsilon,\delta).$
\end{lem}
We state the variational principles  between  the metric mean dimension and the above two types of measure-theoretic $\epsilon$-entropies.

\begin{thm}\label{thm 2.4}{\rm (cf. \cite[Theorems 1.2 and 1.3]{ycz24})}\label{mlk}
Let $(X,d,f)$ be a TDS. Then for any $\delta\in(0,1)$,
\begin{align*}
\overline{\rm mdim}_M(f,X,d)&=\limsup_{\epsilon\to0}\frac{1 }{|\log\epsilon|}\sup_{\mu\in M_f(X)}F(\mu,\epsilon),\\
&=\limsup_{\epsilon\to0}\frac{1 }{|\log\epsilon|}\sup_{\mu\in M_f^e(X)}F(\mu,\epsilon),
\end{align*}
where $F(\mu,\epsilon)\in \{\inf_{{\rm diam}(\xi)<\epsilon}\limits h_\mu(f,\xi), \underline{h}^K_\mu(f,\epsilon,\delta)\}$.

The results are also valid for $\underline{\rm mdim}_M(f,X,d)$ by changing $\limsup_{\epsilon\to 0}$ into $\liminf_{\epsilon\to 0}$.
\end{thm}

\subsection{Specification property of TDSs}
In general, it is difficult to find a real orbit with the desired  dynamical behavior, while one can get an approximate orbit of a point that $\epsilon$-shadows any  specified finite orbit-segments, which is known as  the  \emph{specification property} in TDSs.

\begin{defn}
One says that a TDS $(X,d,f)$ has the specification property if for any $\epsilon>0$, there is a positive integer $m=m(\epsilon)$ such that for any  finite collection of  intervals $I_{j}=[a_{j},b_{j}]\subset\mathbb{N}$ with $a_{j+1}-b_{j}\ge m(\epsilon)$, $j=1,\dots,k-1$,  and any collection  of $k$ points $\{x_{1},\dots,x_{k}\}$ of $X$, there exists  a point  $x\in X$ such that
\begin{equation*}
d(f^{p}x_{j}, f^{p+a_{j}}x)<\epsilon
\end{equation*}
for any $p=0,\cdots,b_{j}-a_{j}$, $j=1,\cdots,k$. 
\end{defn}

 In other words,  the point $x$ $\epsilon$-shadows the  $k$ pieces of finite orbits  of $x_j$:
 \begin{equation*}
 \{ x_j, f(x_j),\ldots,f^{b_j-a_j }(x_j) \},\ j=1,\ldots,k.
 \end{equation*} 
For instance,  every topological mixing locally maximal hyperbolic set, and the full shift  on $[0,1]^{\mathbb{Z}}$  have the specification property.  More examples of infinite entropy systems with  the specification property  can be found in \cite[Section IV]{Backes2023}.

A point $x\in X$ is called \emph{almost periodic} if for any open neighborhood $U$ of $x$, there exists $N\in\mathbb{N}$ such that for every $n\in\mathbb{N}$, there exists an integer $k$ with $n\le k\le n+N$ for which  $f^{k}(x)\in U$.
Denote by $AP(X)$  the set of all almost periodic points of $X$. 

\begin{prop}{\rm \cite[Proposition 2.11]{Hou2023}}\label{prop3}{\rm }
If a TDS $(X,d,f)$ has the specification property, then the almost periodic set AP(X) is dense in $X$.
\end{prop}

\section{Proofs of main results}\label{sec 3}

In this section, we  prove the main results of this paper.

\subsection{Packing metric mean dimension of saturated sets}
To  prove Theorem \ref{thm 1.1}, we divide its proof  into  two parts.

\subsubsection{Upper bound of $\overline{\rm mdim}^P_M(f,G_K,d)$}

\begin{lem}\label{lem4.2}{\rm \cite[Lemma 3.1]{Yuan2024}}
Let $(X,d,f)$ be a TDS and $\epsilon >0$. Let $\{F_n\}$ be a sequence of non-empty  subsets of $X$, and  let $E_{n}$ be an  $(n,\epsilon)$-separated subset of $F_n$ with  the largest cardinality. We  define  
\begin{equation*}
v_{n}:=\frac{1}{\#E_{n}}\sum_{x\in E_{n}}\mathcal{E}_{n}(x).
\end{equation*}
Then for any limit point $\mu$ of the sequence  $\{\nu_n\}$, one has
\begin{equation*}
\limsup_{n\to\infty}\frac{1}{n}\log\#E_{n}\le\inf_{\mathrm{diam}(\xi)<\epsilon}h_{\mu}(f,\xi).
\end{equation*}
\end{lem}
\begin{proof}
It essentially follows from the proof of the upper bound for the  classical variational principle (cf.\cite[Theorem 8.6]{w82}).
\end{proof}

\begin{lem}\label{upper}
Let $(X, d, f)$ be a TDS, and let $K\subset M_{f}(X)$ be a non-empty closed convex subset. Then for any $\epsilon>0$, 
\begin{equation*}
h_{top}^{P}(f,G_{K},\epsilon)\leq h_{top}^{P}(f,G^{K},\epsilon)\le\sup_{\mu\in K}\inf_{\mathrm{diam}(\xi)<\epsilon}h_{\mu}(f,\xi),
\end{equation*}
where $G^K:=\{ x\in X: V_f(x)\subset K\}.$

Consequently, for every $\mu\in M_{f}(X)$ and  $\epsilon>0$,
\begin{equation*}
h_{top}^{P}(f,G_{\mu},\epsilon)\le\inf_{\mathrm{diam}(\xi)<\epsilon}h_{\mu}(f,\xi).
\end{equation*}
\end{lem}

\begin{proof}
It suffices to show the inequality 
$$h_{top}^{P}(f,G^{K},\epsilon)\le\sup_{\mu\in K}\inf_{\mathrm{diam}(\xi)<\epsilon}h_{\mu}(f,\xi).$$

For $n\in\mathbb{N}$  and  $\delta>0$, we put 
\begin{align*}
\mathcal{R}_K(\delta,k):&=\left\{
x\in X:\ \mathcal{E}_n(x)\in B(K,\delta),~ \forall  n \geq k
\right\},\\
\mathcal{Q}_K(\delta,n):&=\left\{
x\in X:\ \mathcal{E}_n(x)\in B(K,\delta)
\right\},
\end{align*}
and  define $$\Theta(K,\epsilon):=\lim_{\delta\to0}\limsup_{n\to\infty}\frac{1}{n}\log s_n(\mathcal{Q}_K(\delta,n),\epsilon),$$ where  $s_n(\mathcal{Q}_K(\delta,n),\epsilon)$ is the largest cardinality of $(n,\epsilon)$-separated sets of  $\mathcal{Q}_K(\delta,n)$.

For every $x\in G^K$, we have $V_f(x)\subset B(K,\delta)$ for every $\delta >0$. Then  the cardinality  of $n$ such that  $ \mathcal{E}_{n}(x)\in M_f(X)\backslash B(K,\delta)$   must be finite.   This shows that  $x\in \mathcal{R}_K(\delta,k)$ for some $k$.
By  $G^K\subset\cup_{k=1}^\infty\mathcal{R}_K(\delta,k)$ for every $\delta>0$, it follows from Proposition \ref{prop 2.1} that
\begin{equation}\label{ppuc}
h_{top}^P(f,G^K,\epsilon)\leq \sup_kh_{top}^P(f,\mathcal{R}_K(\delta,k),\epsilon)\leq \sup_kh_{top}^{UC}(f,\mathcal{R}_K(\delta,k),\epsilon).
\end{equation}
For any $\tau>0$, there exists $\delta_0>0$ such that
\begin{equation*}
\limsup_{n\to\infty}\frac{1}{n}\log s_n(\mathcal{Q}_K(\delta_0,n),\epsilon)<\Theta(K,\epsilon)+\tau.
\end{equation*}
For any $k\geq 1$, it holds that $\mathcal{R}_K(\delta_0,k)\subset\mathcal{Q}_K(\delta_0,n)$  for any $n\geq k$.
Thus, we have
\begin{equation*}
h_{top}^{UC}(f,\mathcal{R}_K(\delta_0,k),\epsilon)\leq\limsup_{n \to \infty}\frac{1}{n}\log s_n(\mathcal{Q}_K(\delta_0,n),\epsilon)<\Theta(K,\epsilon)+\tau.
\end{equation*}
Using \eqref{ppuc} and  letting  $\tau \to 0$, we  conclude that $h_{top}^P(f,G^K,\epsilon)\leq \Theta(K,\epsilon)$. 

 Take a strictly decreasing  sequence $\{\delta_k\}$  with $\delta_k\to0$ and a strictly increasing sequence $\{n_k\}$ with $n_k\to\infty$ such that
\begin{equation*}
s_{n_k}(\mathcal{Q}_K(\delta_k,n_k),\epsilon) > \exp[n_k(\Theta(K,\epsilon)-\gamma) ].
\end{equation*}
Let $F_k$ be a $(n_k,\epsilon)$-separated subset of $\mathcal{R}_K(\delta_k,n_k)$ with the largest cardinality. We set
\begin{equation*}
\mu_k:=\frac{1}{\#F_k}\sum_{x\in F_k}\mathcal{E}_{n_k}(x).
\end{equation*}
Then  $\mu_k \in B(K,\delta_k) $  since $K$ is convex.
Now choose a subsequence $\{n_{k_j}\}$ such that $\mu_{k_j}\to\mu$ as $j\to\infty$. Then $\mu\in K$ since $K$ is  a closed. By Lemma \ref{lem4.2}, we have
\begin{equation*}
\Theta(K,\epsilon)-\gamma\leq  \limsup_{j\to\infty}\frac{1}{n_{k_j}}\log s_{n_{k_j}}(\mathcal{Q}_{k_j}(\delta_{k_j},n_{k_j}),\epsilon) \leq \inf_{ {\rm diam}(\xi)<\epsilon }h_\mu(f,\xi).
\end{equation*}
This  implies that $h_{top}^{P}(f,G^{K},\epsilon)\le\sup_{\mu\in K}\inf_{\mathrm{diam}(\xi)<\epsilon}h_{\mu}(f,\xi)$.
\end{proof}
Thus, by  Lemma \ref{upper} the upper bound of $\overline{\rm mdim}_{M}^{P}(f,G_{K},d)$ is bounded above by
\begin{equation}\label{3.1}
\overline{\rm mdim}_{M}^{P}(f,G_{K},d)
\le\limsup_{\epsilon\to0}\frac{1}{|\log\epsilon|}\sup_{\mu\in K}\inf_{\mathrm{diam}(\xi)<\epsilon}h_{\mu}(f,\xi).
\end{equation}

\subsubsection{Lower bound of $\overline{\rm mdim}^P_M(f,G_K^{\omega,U},d)$} \label{lower}
The challenging direction is  how to estimate  the lower bound of $\overline{\rm mdim}^P_M(f,G_K^{\omega,U},d)$. This shall be done by constructing a \emph{geometric Moran fractal set} that is contained  in $G_K^{\omega,U}$. 

For any sufficiently small $\epsilon^\ast>0$, we  prove that
\begin{equation}\label{3.2}
\sup_{\mu\in K}\inf_{{\rm diam}(\xi)<20\epsilon^\ast}h_{\mu}(f,\xi)\leq h_{top}^P(f,G_K^{\omega,U},\frac{\epsilon^\ast}{12}).
\end{equation}
Then the above inequality yields the desired lower bound.

Without loss of generality, one may assume $H:=\sup_{\mu\in K}\inf_{{\rm diam}(\xi)<20\epsilon^\ast}h_{\mu}(f,\xi)>0$. Let $\zeta >0$ and then take $\mu^\ast\in K$ such that
\begin{equation}\label{to47}
\inf_{\mathrm{diam}(\xi)<20\epsilon^\ast}h_{\mu^\ast}(f,\xi)>H-\zeta.
\end{equation}

Since $K$ is a  connected and compact set, we choose a sequence $\{\mu_{(k,i)}:k\in\mathbb{N},1\le i\le Q_{k}\}\subseteq K$ such that
\begin{align*}
K\subset \bigcup_{i=1}^{Q_{k}}B(\mu_{(k,i)},\frac{1}{k}),\ D(\mu_{(k,i)},\mu_{(k,i+1)})&\le\frac{1}{k},\ 1\leq i<Q_k\\
D(\mu_{(k,Q_{k})},\mu_{(k+1,1)})&\le\frac{1}{k},\ \mu_{(k,Q_{k})}=\mu^\ast.
\end{align*}

Now endow the set $\mathcal{A}=\{(0,0)\}\cup\{(k,j):k\geq 1, 0\leq j\leq Q_k \}$ with the lexicographic order. Denote by $\iota(k,j)$ the position of $(k,j)\in\mathcal{A}$. For $(k,j)\in\mathcal{A}$, we shall use $(k,j)^\ast$ and $(k,j)_\ast$ to represent the one before and the one after $(k,j)$, respectively.

Let $\{\delta_{(k,j)}\}$ be a strictly decreasing sequence indexed by $\mathcal{A}$ with $\delta_{(0,0)}<\zeta$, $\delta_{(k,j)}\to0$ as $(k,j)\to\infty$. We shall simply write $\delta_k=\delta_{(k,Q_k)}$. For  $(k,j)\in\mathcal{A}$, let $\epsilon_{(k,j)}:=\frac{\epsilon^\ast}{2^{\iota(k,j)+2}}$, and $m_{(k,j)}:=m(\epsilon_{(k,j)})$ be  chosen as in the definition of the specification property.

For any non-empty open set $U$, there exist $0<\tilde{\epsilon}<\epsilon^{*}$ and a point $a_{0}\in X$ such that ${\overline{B}(a_{0},\tilde{\epsilon})}\subseteq U$. Set 
\begin{equation*}
\mathcal{C}_{(0,0)}=\{ a_0 \},\ n_{(0,0)}=1,\ 
N_{(0,0)}=1.
\end{equation*}
For any $k\in\mathbb{N}$, there exists $\Delta_k=\{ a_1^k,\ldots,a_{q_k}^k \}\subset AP(X)$,  which is $\epsilon_{(k,0)}$-dense in $X$ by Proposition \ref{prop3}. By the specification property, there exists $a_k\in X$ such that
\begin{equation*}
d\left(
f^{ (i-1)m_{(k,0)}+i-1 }(a_k),a_i^k
\right)<\epsilon_{(k,0)},\ i=1,2,\ldots,q_k.
\end{equation*}
Write $\mathcal{C}_{(k,0)}=\{a_k \}$, $n_{(k,0)}=q_k+(q_k-1)m_{(k,0)}$ and  $N_{(k,0)}=1$.

Let $\mathcal{U}$ be a finite open cover of $X$ with ${\rm diam}(\mathcal{U})\leq 20\epsilon^\ast$ and ${\rm Leb}(\mathcal{U})\geq5\epsilon^\ast$.  The following measure approximation of $\mu^{*}$ is a slight modification of \cite[Lemma 2.1]{Liu2024}, so we omit the proof.

\begin{lem}\label{appro}
For each $k\in\mathbb{N}$, there exists $\nu_k\in M_f(X)$ satisfying
\\
$(1)$ $v_k=\sum_{i=1 }^{j_k}\lambda_i^k\nu_i^k$, where  $\nu_i^k\in M_f^e(X)$,  $\lambda_i^k \in \mathbb{Q}^{+}$ for
every  $ 1\leq i\leq j_k$,   and $\sum_{i=1}^{j_k}\lambda_i^k=1$;\\
$(2)$ one has
\begin{equation}\label{conv}
\inf_{\xi\succ\mathcal{U}}h_{\mu^\ast}(f,\xi)\leq\sum_{i=1}^{j_k}\lambda_i^k \inf_{\xi\succ\mathcal{U}}h_{\nu_i^k}(f,\xi)+\zeta;
\end{equation}\\
$(3)$ $D(\mu^\ast,\nu_k)<\delta_k$, where $D$ is the   compatible metric on $M(X)$ given in \eqref{metr}.
\end{lem}

Since $\nu_i^k$ is ergodic,  by Birkhoff ergodic theorem there exists $N\in\mathbb{N}$ such that the set
\begin{equation*}
Y_i(N)=\left\{
x\in X:D(\mathcal{E}_n(x),\nu_i^k)<\delta_k,\ \forall n>N
\right\}
\end{equation*}
has $\nu_i^k$-measure greater than $1-\zeta$ for all $ 1\leq i\leq j_k$. Applying Lemma \ref{ksmu}, one has
\begin{equation*}
\underline{h}_{\nu_i^k}^K(f,20\epsilon^\ast,\zeta)
\leq
\inf_{\xi\succ\mathcal{U}}h_{\nu_i^k}(f,\xi)
\leq
\underline{h}_{\nu_i^k}^K(f,5\epsilon^\ast,\zeta)<\infty.
\end{equation*}
Now take  an integer $n_k'\in\mathbb{N}$ such that\\
$(1)$ $\lambda_i^kn_k'$ is an integer  for all $1\leq i\leq j_k$, and 
\begin{equation}\label{nk'nk}
\frac{n_k'}{n_k}>\max\left\{
1-\zeta,1-\frac{1}{k}
\right\},
\end{equation}
where $n_k=n_{(k,Q_k)}=n_k'+(j_k-1)m_{(k,Q_k)}$.\\
$(2)$ there exists a subset $E_{k,i}$ that is a $(\lambda_i^kn_k',5\epsilon^\ast)$-separated set of $Y_i(N)$ satisfying 
\begin{equation*}
\#E_{k,i}\geq\exp\left[
\lambda_i^k n_k'\cdot\left(
\underline{h}_{\nu_i^k}^K(f,5\epsilon^\ast,\zeta)-\zeta
\right)
\right]\geq
\exp\left[
\lambda_i^k n_k'\cdot\left(
\inf_{\xi\succ\mathcal{U}}h_{\nu_i^k}(f,\xi)-\zeta
\right)
\right],
\end{equation*}
where the  second inequality is  derived from Lemma \ref{ksmu}.\\
$(3)$ for every $x\in E_{k,i}$,
\begin{equation*}
D(\mathcal{E}_{ \lambda_i^kn_k' }(x),\nu_i^k)<\delta_k.
\end{equation*}

By the specification property, for $k\geq 1$ and each $\underline{x}_{k}=(x_1,x_2,\ldots,x_{j_k})\in E_{k,1}\times E_{k,2}\times\cdots\times E_{k,j_k}$, there exists $y=y(\underline{x}_{k})\in X$ such that
\begin{equation*}
d\left(
f^{T_i+t_i}(y),f^{t_i}(x_i)
\right)<\epsilon_{(k,Q_k)}<\frac{\epsilon^\ast}{2^{k+2}},\ 1\leq i\leq j_k,\ 0\leq t_i\leq \lambda_i^kn_k'-1,
\end{equation*}
where $T_1=0$ and $T_i=\lambda_1^kn_k'+\cdots+\lambda_{i-1}^kn_k'+(i-1)m_k$ for $2\leq i\leq j_k$. For any two distinct $\underline{x}_k=(x_1,x_2,\ldots,x_{j_k})$,  $\underline{x}_k'=(x_1',x_2',\ldots,x_{j_k}')$ in $E_{k,1}\times \cdots\times E_{k,j_k}$, one has $y(\underline{x}_k)\neq y'(\underline{x}_k')$. Indeed, if $x_i\neq x_i'$ for some $1\leq i\leq j_k$, then
\begin{align*}
5\epsilon^\ast<
d_{\lambda_i^kn_k'}\left(
x_i,x_i'
\right)
\leq
d_{\lambda_i^kn_k'}\left(
x_i,f^{T_i}y
\right)
+
d_{\lambda_i^kn_k'}\left(
f^{T_i}y,f^{T_i}y'
\right)
+
d_{\lambda_i^kn_k'}\left(
f^{T_i}y',x_i'
\right)
\\
<\frac{\epsilon^\ast}{2^{k+1}}+d_{\lambda_i^kn_k'}\left(
f^{T_i}y,f^{T_i}y'
\right),
\end{align*}
which implies $d_{n_k}(y,y')>4\epsilon^\ast$.

Let 
\begin{equation*}
\mathcal{C}_{(k,Q_k)}=\left\{y(x_1,x_2,\ldots,x_{j_k}):\
(x_1,x_2,\ldots,x_{j_k})\in E_{k,1}\times E_{k,2}\times\cdots\times E_{k,j_k}
\right\}.
\end{equation*}
Then $\mathcal{C}_{(k,Q_k)}$ is an $(n_k,4\epsilon^\ast)$-separated set and
\begin{align}\label{ckc}
\#\mathcal{C}_{(k,Q_k)}
&= \Pi_{i=1}^{j_k}\#E_{k,i}\geq\Pi_{i=1}^{j_k}\exp\left[
\lambda_i^kn_k'\left(
\inf_{\xi\succ\mathcal{U}}h_{\nu_i^k}(f,\xi)-\zeta
\right)
\right]\nonumber\\
&
=
\exp\left[n_k'\left(\sum_{i=1}^{j_k}\lambda_i^k \inf_{\xi\succ\mathcal{U}}h_{\nu_i^k}(f,\xi)
-\zeta
\right) 
\right]
\overset{\text{by }(\ref{conv})}{\geq}
\exp\left[n_k'
\left(\inf_{\xi\succ\mathcal{U}}h_{\mu^\ast}(f,\xi)-2\zeta\right)
\right]\nonumber\\
&{\geq}
\exp\left[n_k(1-\zeta)
\left(H-3\zeta\right)
\right],
\end{align}
where the last inequality holds since ${\rm diam}(\xi)<  20\epsilon^\ast$ and  (\ref{to47}), (\ref{nk'nk}).

For $n\in\mathbb{N}$ and $\varphi\in C(X)$,  we write $S_n\varphi=\sum_{l=0}^{n-1}\varphi\circ f^i$ as the $n$-th Birkhoff sum of $\varphi$. Then for $y\in\mathcal{C}_{(k,Q_k)}$ and $\varphi\in C(X)$ taking values in $[0,1]$, 
\begin{align*}
&\left|
\int\varphi\dif\mathcal{E}_{n_k}(y)-\int\varphi\dif\left(\sum_{i=1}^{j_k}\lambda_i^k\mathcal{E}_{\lambda_i^kn_k'}(f^{T_i}(y))\right)
\right|\\
\leq&\left|
\frac{1}{n_k}\sum_{i=1}^{j_k}S_{\lambda_i^kn_k'}\varphi(f^{T_i}(y))
-
\frac{1}{n_k'}\sum_{i=1}^{j_k}S_{\lambda_i^kn_k'}\varphi(f^{T_i}(y))
\right|+\frac{n_k-n_k'}{n_k}\|\varphi\|\\
\leq&\left[\left|\frac{1}{n_k'}-\frac{1}{n_k}\right|\cdot n_k'+\frac{n_k-n_k'}{n_k}\right]\cdot\|\varphi\|<\frac{2}{k}\|\varphi\|,
\end{align*}
which shows that $D(\mathcal{E}_{n_k}(y),\sum_{i=1}^{j_k}\lambda_i^k\mathcal{E}_{\lambda_i^kn_k'}f^{T_i}(y))<\frac{2}{k}$. Since the map $x\mapsto\delta_x$ is uniformly continuous, we may assume that for any $x,y \in X$ with $d(x,y)<\frac{\epsilon^\ast}{2^{k+2}}$ implies  $D(\delta_x,\delta_y)<\frac{1}{k}$. 
Since
\begin{equation*}
D\left(\sum_{i=1}^{j_k}\lambda_i^k\mathcal{E}_{\lambda_i^kn_k'}(x_i), \nu_k \right)<\delta_k, \ \ D\left(\nu_k,\mu^\ast \right)<\delta_k,
\end{equation*}
one has that for any $y\in\mathcal{C}_{(k,Q_k)}$,
\begin{equation}\label{ymuast}
D\left(
\mathcal{E}_{n_k}(y),\mu^\ast
\right)
\leq \frac{2}{k}
+
D\left(
\sum_{i=1}^{j_k}\lambda_i^k\mathcal{E}_{\lambda_i^kn_k'}(f^{T_i}(y)),\sum_{i=1}^{j_k}\lambda_i^k\mathcal{E}_{\lambda_i^kn_k'}(x_i)
\right)
+2\delta_k<\frac{3}{k}+2\delta_k.
\end{equation}

Until now,  by virtue of the specification property, we have obtained two  sets $\mathcal{C}_{(k,0)}$ and  $\mathcal{C}_{(k,Q_k)}$ for each $k \geq 1$. We proceed to consider the cases  for $1\leq j <Q_k$.

Since the system has the specification property,  the generic point $G_\mu\neq\emptyset$ holds for all $\mu\in M_f(X)$ \cite{Pfister2007}. Using this fact for any $k\geq 1$ and $1\leq j<Q_k$,  there exist    $x_{k,j}\in G_{\mu_{(k,j)}}$ and $n_{(k,j)}\in\mathbb{N}$ such that
\begin{equation*}
D(\mathcal{E}_{n_{(k,j)}}(x_{k,j}),\mu_{(k,j)})<\delta_{(k,j)},
\end{equation*}
and 
\begin{equation*}
\frac{n_{(k,j)}}{n_{(k,j)}+m_{(k,j)}}>1-\frac{1}{k}.
\end{equation*}
Now set $\mathcal{C}_{(k,j)}=\{x_{k,j} \}$.

With these data $\{\mathcal{C}_{(k,j)}\}$ and $\{n_{(k,j)}\}$, we  are ready  to proceed the next construction:

{\bf Step $1$. Construct the intermediate sets $\mathcal{S}_{(k,j)}$.}

Recall that we have defined $N_{(k,0)}=1$ for any $k\geq 0$.
Now choose a strictly increasing sequence of integers $\{ N_{(k,j)} \}_{(k,j)\in\mathcal{A},~ j\neq 0}$ such that $N_{(1,1)}=1$ and
\begin{equation}\label{chongfubi}
\begin{split}
\sum_{0\leq k'\leq k+1}(n_{(k',0)}+m_{(k',0)})+n_{(k,j)_\ast}+m_{(k,j)_\ast}&\leq
\frac{1}{k}\sum_{(k',j')\leq(k,j)}[N_{(k',j')}\cdot(n_{(k',j')}+m_{(k',j')})] ,\\
\sum_{(k',j')\leq(k,j)}[N_{(k',j')}\cdot(n_{(k',j')}+m_{(k',j')})]
&\leq\frac{N_{(k,j)_\ast}\cdot n_{ (k,j)_\ast }}{k}
,\\
n_kN_{(k,Q_k)}&\geq(1-\zeta)\sum_{ 0\leq j\leq Q_k }[ (n_{(k,j)}+m_{(k,j)})N_{(k,j)}].
 \end{split}
\end{equation}

For $(k,0)\in\mathcal{A}$, we set $\mathcal{S}_{(k,0)}:=\mathcal{C}_{(k,0)}$ and $c_{k,0}:=n_{(k,0)}$. We divide the following two cases to define the rest $\mathcal{S}_{(k,j)}$  indexed  by $j\neq 0$.

Case $1$. $1\leq j<Q_k$.

Recall $\mathcal{C}_{(k,j)}=\{ x_{k,j} \}$.
By the specification property again, there exists $y_{k,j}\in X$ such that the pieces of orbits $\{ x_{k,j},\cdots,f^{n_{(k,j)-1}}(x_{(k,j)})  \}$ repeating $N_{(k,j)}$ times, can be $\epsilon_{(k,j)}$-traced by $y_{(k,j)}$ with  the gap $m_{(k,j)}$. 
Now we set $\mathcal{S}_{(k,j)}=\{y_{k,j}\}$, and $c_{k,j}=n_{(k,j)}N_{(k,j)}+(N_{(k,j)}-1)m_{(k,j)}$.

Case $2$. $j=Q_k$.

Enumerate the points in $\mathcal{C}_{(k,Q_k)}$ as 
\begin{equation*}
\mathcal{C}_{(k,Q_k)}=\{ y_i^k: i=1,2,\ldots,\#\mathcal{C}_{(k,Q_k)} \}.
\end{equation*}
For any $(i_1,\ldots,i_{N_{(k,Q_k)}})\in\{ 1,2,\ldots,\#\mathcal{C}_{(k,Q_k)} \}^{ N_{(k,Q_k)} }$, let $y(i_1,\ldots,i_{N_{(k,Q_k)}})\in X$ be a point given by the specification property, such that the orbit $\epsilon_{(k,Q_k)}$-shadows, with the gap $m_{(k,Q_k)}$, the pieces of orbits $\{ y_{i_l}^k,f(y_{i_l}^k),\ldots,f^{ n_k-1 }(y_{i_l}^k) \}$, $l=1,2,\ldots,N_{(k,Q_k)}$. We write
\begin{equation*}
\mathcal{S}_{(k,Q_k)}
=\left\{
y=y(i_1,\ldots,i_{N_{(k,Q_k)}})
:
(i_1,\ldots,i_{N_{(k,Q_k)}})\in\{ 1,2,\ldots,\#\mathcal{C}_{(k,Q_k)} \}^{ N_{(k,Q_k)} }
\right\}
\end{equation*}
 and  $c_{k,Q_k}=n_k\cdot N_{(k,Q_k)}+(N_{(k,Q_k)}-1)m_{(k,Q_k)}$. Since
$\mathcal{C}_{(k,Q_k)}$ is an $(n_k,4\epsilon^\ast)$-separated set, this implies that $\mathcal{S}_{(k,Q_k)}$ is a $(c_{k,Q_k},3\epsilon^\ast)$-separated set. Then, by (\ref{ckc}) and (\ref{chongfubi}),  the cardinality of  $\mathcal{S}_{(k,Q_k)}$ is bounded blew by
\begin{equation}\label{geshu}
\#\mathcal{S}_{(k,Q_k)}=\#\mathcal{C}_{(k,Q_k)}^{ N_{(k,Q_k)} }\geq\exp\left[
\sum_{ 0\leq j\leq Q_k }[ (n_{(k,j)}+m_{(k,j)})N_{(k,j)}]\cdot(1-\zeta)^2\cdot(H-3\zeta )
\right]. 
\end{equation}

Similar to the arguments in (\ref{ymuast}), one has
\begin{equation*}
D(\mathcal{E}_{c_{k,j }}(y_{k,j}),\mu_{(k,j)})<\frac{2}{k}+\delta_{(k,j)},\ \forall\ y_{k,j}\in\mathcal{S}_{(k,j)}~ \text{with}~ k\geq 1\ \text{and }1\leq j<Q_k,
\end{equation*}
and
\begin{equation*}
D(\mathcal{E}_{c_{k,Q_k}}(y),\mu^\ast)<\frac{5}{k}+2\delta_k,
\ \forall\ y\in\mathcal{S}_{(k,Q_k)}. 
\end{equation*}
Hence, noting that $\mu_{(k,Q_k)}=\mu^{*}$, in an unified  manner we always have
\begin{equation}\label{yukj}
D(\mathcal{E}_{c_{k,j }}(y),\mu_{(k,j)})<\frac{5}{k}+2\delta_{(k,j)},\ \forall\ y\in\mathcal{S}_{(k,j)}\ \text{with }k\geq 1,\ 1\leq j\leq Q_k.
\end{equation}

{\bf Step $2$. Construct the $\mathbb{F}_{(k,j)}$ and the Moran set $\mathbb{F}$.}

This step can be regarded  as an analogue of  the  construction of  the middle-third Cantor set of $[0,1]$.

Let $\mathcal{I}_{(0,0)}=\mathcal{S}_{(0,0)}$ and $M_{(0,0)}=n_{(0,0)}$. Suppose that we have already constructed the set $\mathcal{I}_{(k,j)^\ast}$, and  then we will describe how to construct $\mathcal{I}_{(k,j)}$. Let
\begin{align*}
M_{(k,j)}&=M_{ (k,j)^\ast }+m_{(k,j)}+c_{k,j};\\
M_{(k,j)}^i&=M_{ (k,j)^\ast }+i(m_{(k,j)}+n_{(k,j)}),\ 0\leq i\leq N_{(k,j)}.
\end{align*}
Note that $M_{(k,j)}^{N_{(k,j)}}=M_{(k,j)}$ and  the number of $\mathcal{S}_{(k,Q_k)}$ given in  (\ref{geshu})  has the estimation:
\begin{equation}\label{equ 3.11}
\#\mathcal{S}_{(k,Q_k)}\geq\exp\left[
(M_{(k,Q_k)}-M_{(k-1,Q_{k-1}) })(1-\zeta)^2(H-3\zeta)
\right].
\end{equation}
For $x\in\mathcal{I}_{(k,j)^\ast }$ and $y\in\mathcal{S}_{(k,j)}$,  by the specification property let $z=z(x,y)\in X$ be some point such that
\begin{equation*}
d_{ M_{ (k,j)^\ast } }(x,z)<\epsilon_{(k,j)},\ d_{c_{k,j}}(y,f^{ M_{(k,j)^\ast}+m_{(k,j)} }(z))<\epsilon_{(k,j)}.
\end{equation*}
We set $\mathcal{I}_{(k,j)}=\{ z(x,y): x\in\mathcal{I}_{(k,j)^\ast }, y\in\mathcal{S}_{(k,j)} \}$.
For $x\in\mathcal{I}_{(k,Q_k)^\ast}$ and  two distinct $y, y'\in \mathcal{S}_{(k,Q_k)}$, since $\mathcal{S}_{(k,Q_k)}$ is a  $(c_{k,Q_k},3\epsilon^\ast)$-separated set, one has
\begin{align*}
d_{ M_{(k,Q_k)} }(z(x,y),z(x,y'))>2\epsilon^\ast   ,
\end{align*}
which leads to
\begin{align}\label{equ 3.12}
\#\mathcal{I}_{(k,j)}
=\#\mathcal{I}_{(k,j)^\ast }\cdot\#\mathcal{S}_{(k,j)}
=\Pi_{(k',j')\leq(k,j)}\#\mathcal{S}_{(k',j')}=
\Pi_{(k',Q_{k'})\leq(k,j)}\#\mathcal{S}_{(k',Q_{k'})}.
\end{align}
Moreover, according to (\ref{yukj}), for $z\in\mathcal{I}_{(k,j)}$ with $j\neq 0$, we have
\begin{equation}\label{iapp}
D( \mathcal{E}_{c_{k,j}}( f^{ M_{(k,j)^\ast} }(z) ),\mu_{(k,j )})<\frac{6}{k}+2\delta_{(k,j)}.
\end{equation}

Now set
\begin{equation*}
\mathbb{F}_{(k,j)}=\bigcup_{ x\in\mathcal{I}_{(k,j)} }\overline{B}_{ M_{(k,j)}}(x,2\epsilon_{(k,j)}).
\end{equation*}
Then for each $(k,j)\in\mathcal{A}$, the following  statements are   satisfied:\\
$(1)$ $\overline{B}_{ M_{(k,j)}}(x,2\epsilon_{(k,j)})\cap \overline{B}_{ M_{(k,j)}}(x',2\epsilon_{(k,j)})=\emptyset$ holds for any distinct $x, x'\in \mathcal{I}_{(k,j)}$;\\
$(2)$ if $z\in \mathcal{I}_{(k,j)}$ descends from $x\in \mathcal{I}_{(k,j)^\ast}$, one has $$\overline{B}_{ M_{(k,j)} }(z,2\epsilon_{(k,j)})\subset \overline{B}_{ M_{(k,j)^\ast} }(x,2\epsilon_{(k,j)^\ast}).$$ 
This implies that  $\{\mathbb{F}_{(k,j)}: (k,j)\in\mathcal{A}\}$ is a sequence of decreasing non-empty closed subsets of $X$. Let 
\begin{equation*}
\mathbb{F}=\bigcap_{(k,j)\in\mathcal{A}}\mathbb{F}_{(k,j)}=\bigcap_{k\in\mathbb{N}}\mathbb{F}_{ (k,Q_k) }.
\end{equation*}

\begin{lem}
The fractal set $\mathbb{F}$ is contained in  $G_K^{\omega,U}$,  i.e.,
\begin{equation*}
\mathbb{F}\subset G_K^{\omega,U}= G_K\cap U\cap \{ x\in X: \omega_f(x)=X \}.
\end{equation*}
\end{lem}
\begin{proof}[Proof of {Lemma 3.4}]
By the construction of $\mathbb{F}$,   we know $$\mathbb{F}\subset\mathbb{F}_{(0,0)}=\overline{B}_{n(0,0)}(a_0,\frac{\epsilon^\ast}{8})\subset {\overline{B}(a_0,\tilde{\epsilon})}\subset U.$$
For $p\in\mathbb{F}$,  we have $p\in \mathbb{F}_{(k,0)}$ for any $k\geq 1$. Hence, for any $a_i^k\in\Delta_k$, $1\leq i\leq q_k$, there exists $t\in[M_{(k,0)^\ast},M_{(k,0)}]$ such that $d(f^t(p),a_i^k)<4\epsilon_{(k,0)}$. This  implies that the orbit of $p$ is $5\epsilon_{(k,0)}$-dense in $X$.
By the arbitrariness of $k\geq 1$ and $p\in\mathbb{F}$, we have  $$\mathbb{F}\subset\{ x\in X: \omega_f(x)=X \}.$$ 

Let $p\in\mathbb{F}$. For sufficiently large $n\in\mathbb{N}$, to estimate $D(\mathcal{E}_n(p), K)$ we split into four cases to discuss.

Case $(1)$:  There exists $k>1$ such that $M_{(k+1,0)}\leq n<M_{(k+1,1)}^1$.

Then there is some $z\in\mathcal{I}_{(k,Q_k)}$ such that $p\in \overline{B}_{ M_{(k,Q_k)} }(z,\frac{\epsilon_{(k,Q_k)}}{2})$, and
\begin{align*}
	D( \mathcal{E}_n(p), \mu^\ast)\leq 
	&\frac{M_{ (k,Q_k)^\ast }}{n}D( \mathcal{E}_{ M_{ (k,Q_k)^\ast } }(p),\mu^\ast )
	+
	\frac{n-M_{(k,Q_k) }}{n}D (
	\mathcal{E}_{ n-M_{(k,Q_k) } }(f^{ M_{ (k,Q_k) } } (p)),\mu^\ast
	)
	\\
	&+
	\frac{c_{k,Q_k}+m_{(k,Q_k)} }{n}D( \mathcal{E}_{ c_{k,Q_k}+m_{(k,Q_k)} }(f^{ M_{ (k,Q_k)^\ast } } (p)),\mu^\ast )
	.
\end{align*}
Noticing that both $\frac{M_{ (k,Q_k)^\ast }}{n}$ and $\frac{n-M_{(k,Q_k) }}{n}$ are smaller than $\frac{1}{k}$, and using formula (\ref{iapp}), we have $D( \mathcal{E}_{ c_{k,Q_k}+m_{(k,Q_k)} }f^{ M_{ (k,Q_k)^\ast } } (p),\mu^\ast )<\frac{7}{k}+2\delta_{ (k,Q_k) }$. Thus,
\begin{equation*}
	D( \mathcal{E}_n(p), K)\leq \frac{9}{k}+2\delta_{ (k,Q_k) }.
\end{equation*}

Case $(2)$: There exists $(k,j)\in\mathcal{A}$ with $j\neq 0$ such that $M_{(k,j)}\leq n<M_{(k,j)_\ast}^1$.

Then there is some $z\in\mathcal{I}_{(k,j)}$ such that $p\in \overline{B}_{ M_{(k,j)} }(z,\frac{\epsilon_{(k,j)}}{2})$, and
\begin{align*}
	D( \mathcal{E}_n(p), \mu_{(k,j)})\leq& 
	\frac{M_{ (k,j)^\ast }}{n}D( \mathcal{E}_{ M_{ (k,j)^\ast } }(p),\mu_{(k,j)} )
	+
	\frac{n-M_{(k,j) }}{n}D (
	\mathcal{E}_{ n-M_{(k,j) } }(f^{ M_{ (k,j) } } (p)),\mu_{(k,j)}
	)\\
	&+
	\frac{c_{k,j}+m_{(k,j)}}{n}D( \mathcal{E}_{ c_{k,j}+m_{(k,j)} }(f^{ M_{ (k,j)^\ast } } (p)),\mu_{(k,j)} )\\
	.
\end{align*}
The similar argument  as in  Case $(1)$ yields  that $$D( \mathcal{E}_n(p), K)\leq \frac{9}{k}+2\delta_{ (k,j) }.$$

Case $(3)$: There exist $(k,1)\in\mathcal{A}$ and $i>1$ such that $M_{(k,1)}^i\leq n<M_{(k,1)}^{i+1}$.

One has
\begin{align*}
	D( &\mathcal{E}_n(p),\mu_{ (k,1)} )
	\leq
	\frac{ M_{ (k-1,Q_{k-1})^\ast  } }{n}
	D( \mathcal{E}_{M_{ (k-1,Q_{k-1})^\ast  } } (p), \mu_{ (k,1)} )\\
	&+
	\frac{ c_{ k-1,Q_{k-1} }+m_{ (k-1,Q_{k-1}) } }{n}
	D( \mathcal{E}_{c_{ k-1,Q_{k-1}}+m_{ (k-1,Q_{k-1}) }}(f^{M_{ (k-1,Q_{k-1})^\ast}} (p)), \mu_{ (k,1)} )\\
	&+
	\frac{c_{k,0}+m_{(k,0)}}{n}
	D( \mathcal{E}_{c_{k,0}+m_{(k,0)}}( f^{M_{ (k-1,Q_{k-1})}}(p)  )   , \mu_{ (k,1)} )\\
	&+\frac{ n_{(k,1)}+m_{(k,1)}  }{n}\sum_{l=0}^{i-1}
	D(\mathcal{E}_{n_{(k,1)}+m_{(k,1)}}( f^{ M_{(k,1)}^{l} } (p) )
	, \mu_{ (k,1)} )\\
	&+
	\frac{n-M_{(k,1)}^i}{n}
	D(\mathcal{E}_{n-M_{(k,1)}^i }( f^{M_{(k,1)}^i } (p)  )
	, \mu_{ (k,1)} ).
\end{align*}
Note that
\begin{equation*}
	\frac{ M_{ (k-1,Q_{k-1})^\ast  } }{n}<\frac{1}{k-1},\
	\frac{c_{k,0}+m_{(k,0)}}{n}<\frac{1}{k},\
	\frac{n-M_{(k,1)}^i}{n}<\frac{1}{k},
\end{equation*}
\begin{equation*}
	D(\mathcal{E}_{n_{(k,1)}+m_{(k,1)}}( f^{ M_{(k,1)}^{l} } (p) )
	, \mu_{ (k,1)} )\leq\frac{5}{k}+2\delta_{(k,1)},
\end{equation*}
and $D( \mathcal{E}_{c_{ k-1,Q_{k-1}}+m_{ (k-1,Q_{k-1}) }}(f^{M_{ (k-1,Q_{k-1})^\ast}} (p)), \mu_{ (k,1)} )$ is smaller than
\begin{equation*}
	D( \mathcal{E}_{c_{ k-1,Q_{k-1}}+m_{ (k-1,Q_{k-1})}}(f^{M_{ (k-1,Q_{k-1})^\ast}} (p)), \mu^\ast )+D(\mu^\ast, \mu_{ (k,1)})
	\leq \frac{8}{k-1}+2\delta_{ (k-1,Q_{k-1}) }.
\end{equation*}
This gives us 
\begin{equation*}
	D( \mathcal{E}_n(p), K)\leq\frac{16}{k-1}+4\delta_{(k-1,0)}.
\end{equation*}

Case $(4)$: There exist $(k,j)\in\mathcal{A}$ with $j>1$ and $i>1$  such that $M_{(k,j)}^i\leq n<M_{(k,j)}^{i+1}$.

In this case, using the expression
\begin{align*}
	\mathcal{E}_n(z)
	&=\frac{M_{(k,j-1)^\ast } }{n}\mathcal{E}_{ M_{(k,j-1)^\ast } }(p)
	+\frac{c_{k,j-1}+m_{(k,j-1)}}{n}
	\mathcal{E}_{ c_{k,j-1}+m_{(k,j-1)} }( f^{ M_{(k,j-1)^\ast} }(p) )
	\\
	&+\frac{n_{(k,j)}+m_{(k,j)}}{n}
	\sum_{l=0}^{i-1 }\mathcal{E}_{  
		n_{(k,j)}+m_{(k,j)}
	}( f^{ M_{(k,j)}^l }(p) )+\frac{n- M_{(k,1)}^i}{n}\mathcal{E}_{ n-M_{(k,1)}^i }(f^{ M_{(k,1)}^i  }(p)),
\end{align*}
and repeating the similar argument as above, one can obtain
\begin{equation*}
	D( \mathcal{E}_n(p), K)\leq\frac{15}{k}+4\delta_{(k,j)}.
\end{equation*}

To sum up, whatever case we consider, by letting $n\to\infty$ we get $V_f(p)\subset K$. Based on the above arguments, one can  also show  that the sequences $\{ \mathcal{E}_n(p) \}_{n\in\mathbb{N}}$ and $\{ \mu_{(k,j)} \}_{ (k,j)\in\mathcal{A} }$ have the same limit points. Notice that $K=\overline{ \{  \mu_{(k',j')}: (k',j')>(k,j)\}}$  for any fixed $(k,j)\in\mathcal{A}$. Consequently, we have $V_f(p)=K$, and hence $\mathbb{F}\subset G_K$.
\end{proof}

{\bf Step $3$. Construct the measure $\rho$ supported on $\mathbb{F}$.}

Now we define a probability measure $\rho$ on $\mathbb{F}$ satisfying  the $\rho$-measure of some appropriate dynamical balls decays exponentially fast along  the sequence $\{ M_k \}$, where $M_k:=M_{(k,Q_k)}$. Then we can invoke the Entropy  Distribution Principle stated in  Lemma \ref{lem5}  to obtain a lower bound of  $\epsilon$-packing topological entropy of $\mathbb{F}$.

Define
\begin{equation*}
\rho_k=\frac{1}{\# \mathcal{I}_{(k,Q_k)} }\sum_{x\in \mathcal{I}_{(k,Q_k)}}\delta_x.
\end{equation*}
For simplicity, we write  $\mathcal{I}_k$, $\mathbb{F}_k$  instead of  $\mathcal{I}_{(k,Q_k)}$, $\mathbb{F}_{(k,Q_k)}$ respectively.
Since $M(X)$ is compact, we may assume that $\rho_k\to\rho\in M(X)$ as $k\to\infty$. It is clear that $\rho_k( \mathbb{F}_k )=1$ for every $k$, and
for any $k>k_0$, one has $\rho_k(\mathbb{F}_{k_0})=1$ since $\mathbb{F}_k\subset\mathbb{F}_{k_0}$. Thus, $\rho( \mathbb{F}_{k_0} )\geq \limsup_{k\to\infty}\rho_k( \mathbb{F}_{k_0} )=1$. This implies that $\rho(\mathbb{F})=1$. Recall that $M_{(0,0)}=n_{(0,0)}=1$. By (\ref{equ 3.11}) and \eqref{equ 3.12}, we have
\begin{equation*}
\#\mathcal{I}_k\geq\exp\left[
(M_k -1)\cdot(1-\zeta)^2\cdot( H-3\zeta )
\right]. 
\end{equation*}

\begin{lem}\label{lem 3.5}
Let $k_0\geq 1$, $z\in X$ and $B_{M_{k_0}}(z,\frac{\epsilon^\ast}{2})$ be a Bowen ball  that intersects $\mathbb{F}$. Then,  for any $k>k_0$, one has
\begin{equation*}
\rho_{k}\left(B_{M_{k_0}}\left(z,\frac{\epsilon^\ast}{2}\right) \right)\leq\frac{1}{\#\mathcal{I}_{k_0}}.
\end{equation*}
\end{lem}
\begin{proof}[Proof of Lemma 3.5]
Since $B_{M_{k_0}}(z,\frac{\epsilon^\ast}{2})\cap \mathbb{F}\not=\emptyset$, the set $B_{M_{k_0}}(z,\frac{\epsilon^\ast}{2})\cap \mathbb{F}_{k_0}$ is also non-empty by the construction of $\mathbb{F}$ and  $\#(\mathbb{F}_{k_0}\cap B_{M_{k_0}}(z,\frac{\epsilon^\ast}{2}))=1$. Notice that  $\mathbb{F}_{k_0+1}$ descends from $\mathbb{F}_{k_0}$. Then  one has
\begin{equation*}
\#(\mathbb{F}_{k_0+1}\cap B_{M_{k_0}}(z,\frac{\epsilon^\ast}{2}))\leq\#\mathcal{S}_{(k_0+1,Q_{ k_0+1 }) }, 
\end{equation*}
 and hence
\begin{equation*}
\rho_{k_0+1}\left(
B_{M_{k_0}}(z,\frac{\epsilon^\ast}{2})
\right)\leq\frac{ \#\mathcal{S}_{(k_0+1,Q_{ k_0+1 }) }}{\#\mathcal{I}_{k_0+1}}
=
\frac{1}{\#\mathcal{I}_{k_0}}.
\end{equation*}
In general, for any $k>k_0$, one has
\begin{equation*}
\rho_{k}\left(
B_{M_{k_0}}(z,\frac{\epsilon^\ast}{2})
\right)\leq
\frac{ \Pi_{k_0+1\leq k'\leq k}\#\mathcal{S}_{ (k',Q_{k'}) } }{\#\mathcal{I}_{k}}=\frac{1}{\#\mathcal{I}_{k_0}}.
\end{equation*}
\end{proof}

By Lemma \ref{lem 3.5}, for sufficiently large $k_0$ and any Bowen ball $B_{ M_{k_0} }(z,\frac{\epsilon^\ast}{2})$ with $z\in\mathbb{F}$, since $B_{ M_{k_0} }(z,\frac{\epsilon^\ast}{2})$ is open,  one has
\begin{align*}
\rho\left(
B_{ M_{k_0} }(z,\frac{\epsilon^\ast}{2})
\right)
&
\leq\liminf_{k\to\infty}\rho_k\left(
B_{ M_{k_0} }(z,\frac{\epsilon^\ast}{2})
\right)\leq\frac{1}{\#\mathcal{I}_{k_0}}\\
&\leq
\exp\left[
-(M_{k_0}-1)(1-\zeta)^2\cdot( H-3\zeta )
\right]\\
&\leq
\exp\left[
-M_{k_0}(1-\zeta)^3\cdot( H-3\zeta )
\right]
\end{align*}
Using Lemma \ref{lem5}, one has
\begin{equation*}
(1-\zeta)^3\cdot(H-3\zeta)\leq h_{top}^P(f,\mathbb{F},\frac{\epsilon^\ast}{12})\leq h_{top}^P(f,G_K^{\omega,U},\frac{\epsilon^\ast}{12}).
\end{equation*}
The arbitrariness of $\zeta$ implies that
\begin{equation*}
 \sup_{\mu\in K}\inf_{{\rm diam}(\xi)<20\epsilon^\ast}h_{\mu}(f,\xi)\leq h_{top}^P(f,G_K^{\omega,U},\frac{\epsilon^\ast}{12}).
\end{equation*}
Dividing  in both sides of the above inequality by $|\log\epsilon^\ast|$, and then taking $\limsup_{\epsilon^\ast\to0}$, we have
\begin{equation}\label{3. 12}
\limsup_{\epsilon\to0}\frac{1}{|\log\epsilon|} \sup_{\mu\in K}\inf_{{\rm diam}(\xi)<\epsilon}h_{\mu}(f,\xi)\leq \overline{\rm mdim}_M^P(f,G_K^{\omega,U},d).
\end{equation}

Finally, by (\ref{3.1}) and (\ref{3. 12})  we get
\begin{align*}
\overline{\rm mdim}_M^P(f,G_K^{\omega,U},d)=\overline{\rm mdim}_M^P(f,G_K,d)=\limsup_{\epsilon\to0}\frac{1}{|\log\epsilon|} \sup_{\mu\in K}\inf_{{\rm diam}(\xi)<\epsilon}h_{\mu}(f,\xi),
\end{align*}
which proves  Theorem \ref{thm 1.1}.

\subsection{Bowen metric mean dimension of  saturated sets}
In this subsection, we prove Theorems \ref{thm 1.2} and \ref{thm 1.3}. 

\subsubsection{Upper bound of $\overline{\rm mdim}^B_M(f,G_K,d)$}

The following lemma allows us to get an upper bound of $\overline{\rm mdim}^B_M(f,G_K,d)$.
\begin{lem}{\rm \cite[Lemma 3.3,(3)]{Yuan2024}}\label{lem 3.6}
Let $(X,d,f)$ be a  TDS, and let $K\subset M_{f}(X)$ be a non-empty closed  subset.

\begin{itemize}
\item  [(1)] For any $\epsilon>0$,  we have $$h_{top}^{B}(f,\{x\in X:V_f(x)\cap K\not=\emptyset\},\epsilon)\leq\sup_{\mu\in K}\inf_{\mathrm{diam}(\xi)<\epsilon}h_{\mu}(f,\xi).$$
\item  [(2)] For any $\epsilon>0$, we have 
\begin{equation*}
h_{top}^{B}(f,G_{K},\epsilon)\le\inf_{\mu\in K}\inf_{\mathrm{diam}(\xi)<\epsilon}h_{\mu}(f,\xi).
\end{equation*}
\end{itemize}
\end{lem}

Thus, by  Lemma  \ref{lem 3.6} we have
\begin{equation}\label{3.13}
\overline{\rm mdim}_{M}^{B}(f,G_{K},d)
\le\limsup_{\epsilon\to0}\frac{1}{|\log\epsilon|}\inf_{\mu\in K}\inf_{\mathrm{diam}(\xi)<\epsilon}h_{\mu}(f,\xi).
\end{equation}

\subsubsection{Lower bound of $\overline{\rm mdim}^B_M(f,G_K^{\omega,U},d)$}

The proof for  the lower bound of $\overline{\rm mdim}^B_M(f,G_K^{\omega,U},d)$  is similar to that in subsection \ref{lower}. So we  only give a sketch for the proof. 

For any sufficiently small $\epsilon^\ast>0$, we show
\begin{equation}\label{bl}
\inf_{\mu\in K}\inf_{{\rm diam}(\xi)<20\epsilon^\ast}h_{\mu}(f,\xi)
\leq
h_{top}^B(f,G_K^{\omega,U},\frac{\epsilon^\ast}{2}).
\end{equation}

Fix $\zeta >0$ and pick  a sequence $\{ \mu_{(k,i)}: k\geq 1, 1\leq i\leq Q_k \}\subset K$ of measures in $K$ such that
\begin{equation*}
K\subset\cup_{i=1}^{Q_k}B( \mu_{(k,i)},\frac{1}{k} ),\ 
D(\mu_{(k,i)},\mu_{(k,i+1)})<\frac{1}{k},\ 
D(\mu_{(k,Q_k)},\mu_{(k+1,1)})<\frac{1}{k}.
\end{equation*}
Put $\mathcal{A}'=\{ (k,i):k\geq 1, 1\leq i\leq Q_k \}$, $\mathcal{A}=\{(k,0):k\geq 0\}\cup\mathcal{A}'$, and endow these sets with the lexicographic order again. Using the fixed notation in subsection \ref{lower}, we have the data $\{\delta_{(k,j)}\}_{ (k,j)\in\mathcal{A} }$, $\{\epsilon_{(k,j)}\}_{ (k,j)\in\mathcal{A} }$ and $\{m_{(k,j)}\}_{ (k,j)\in\mathcal{A} }$.

Assume that $H:=\inf_{\mu\in K}\inf_{{\rm diam}(\xi)<20\epsilon^\ast}h_{\mu}(f,\xi)>0$. Then one has
\begin{equation*}
\inf_{ {\rm diam}(\xi)<20\epsilon^\ast }h_{\mu_{(k,i)}}(f,\xi)>H-\zeta,\ \forall (k,i)\in\mathcal{A}'.
\end{equation*}
Let $\mathcal{U}$ be an open cover of $X$ with ${\rm diam}(\mathcal{U})\leq20\epsilon^\ast$, and  ${\rm Leb}(\mathcal{U})\geq5\epsilon^\ast$.
For $(k,i)\in\mathcal{A}'$, by Lemma \ref{appro} let $\nu_{(k,i)}=\sum_{l=1}^{p_{k,i}}\lambda_{k,i,l}\nu_l^{k,i}$ be the finite convex combination with rational coefficients of  $f$-ergodic probability measures  such that
\begin{equation*}
D( \mu_{(k,i)},\nu_{(k,i)} )<\delta_{(k,i)},\ \ \inf_{ \xi\succ\mathcal{U} }
h_{\mu_{(k,i)}}(f,\xi)\leq \sum_{l=1}^{p_{k,i}}\lambda_{k,i,l}\inf_{ \eta\succ\mathcal{U} }
h_{\nu_l^{k,i}}
(f,\eta)+\zeta.
\end{equation*}

Repeating the previous argument as in subsection \ref{lower},  we  choose an integer $n_{(k,i)}'$ such that the following conditions are satisfied:\\
$(1)$ for any $1\leq l\leq p_{k,i}$, $\lambda_{k,i,l}n_{(k,i)}'$ is an integer and
\begin{equation*}
\frac{n_{(k,i)}'}{n_{(k,i)}+m_{(k,i)}}>\max\{1-\zeta,1-\frac{1}{k} \},
\end{equation*}
where $n_{(k,i)}:=n_{(k,i)}'+(p_{k,i}-1)m_{(k,i)}$;\\
$(2)$ there exist subsets $E_{k,i,l}$, $1\leq l\leq p_{k,i}$, which are $(\lambda_{k,i,l}n_{(k,i)}',5\epsilon^\ast)$-separated satisfying
\begin{equation*}
\#E_{k,i,l}\geq\exp\left[
\lambda_{k,i,l}n_{(k,i)}'\cdot
\left(
\inf_{\eta\succ\mathcal{U}}h_{\nu_l^{k,i}}(f,\eta)-\zeta
\right)
\right];
\end{equation*}
$(3)$ for any $x\in E_{k,i,l}$,
\begin{equation*}
D( \mathcal{E}_{ \lambda_{k,i,l}n_{(k,i)}' }(x),\nu_l^{k,i} )<\delta_{(k,i)}.
\end{equation*}

Let 
\begin{equation*}
\mathcal{C}_{(k,i)}=\left\{
y=y(x_1,\ldots,x_{p_{k,i}}):  (x_1,\ldots,x_{p_{k,i}})\in E_{k,i,1}\times\cdots\times E_{k,i,p_{k,i}}
\right\},
\end{equation*}
where $y(x_1,\ldots,x_{p_{k,i}})$ is the point that $\epsilon_{(k,i)}$-shadows the pieces of orbits 
\begin{equation*}
\{ x_l,\ldots,f^{\lambda_{k,i,l}n_{(k,i)}'}(x_l) \},\ 1\leq l\leq p_{k,i},
\end{equation*} 
with gap $m_{(k,i)}$.
Then $\mathcal{C}_{(k,i)}$ is a  $(n_{(k,i)},4\epsilon^\ast)$-separated set  with
\begin{equation*}
\#\mathcal{C}_{(k,i)}\geq \exp\left[
n_{(k,i)}(1-\zeta)\cdot(H-2\zeta)
\right].
\end{equation*}

For $k\geq 0$, let $\mathcal{C}_{(k,0)}$, $n_{(k,0)}$, $N_{(k,0)}$ be defined as in subsection \ref{lower}. Choose a strictly increasing sequence $\{N_{(k,i)} \}_{(k,i)\in\mathcal{A}'}$ such that $N_{(1,1)}=1$ and
\begin{equation*}
\begin{split}
\sum_{0\leq k'\leq k+1}(n_{(k',0)}+m_{(k',0)})+n_{(k,j)_\ast}+m_{(k,j)_\ast}&\leq
\frac{1}{k}\sum_{(k',j')\leq(k,j)}[N_{(k',j')}\cdot(n_{(k',j')}+m_{(k',j')})] ,\\
\sum_{(k',j')\leq(k,j)}[N_{(k',j')}\cdot(n_{(k',j')}+m_{(k',j')})]
&\leq\frac{N_{(k,j)_\ast}\cdot n_{ (k,j)_\ast }}{k}.
\end{split}
\end{equation*}

Set $\mathcal{S}_{(k,0)}=\mathcal{C}_{(k,0)}$, $c_{k,0}=n_{(k,0)}$ for $k\geq 0$. For every  $(k,i)\in\mathcal{A}'$ and any tuple of points $(y_{k,i,1},y_{k,i,2},\ldots,y_{k,i,N_{(k,i)}})\in\mathcal{C}_{(k,i)}^{ N_{(k,i)} }$, there exists a point $z=z(y_{k,i,1},y_{k,i,2},\ldots,y_{k,i,N_{(k,i)}})$ that $\epsilon_{(k,i)}$-shadows the pieces of orbits
\begin{equation*}
\{y_{k,i,l},f(y_{k,i,l}),\ldots,f^{ n_{(k,i)} -1}(y_{k,i,l}) \},\ l=1,\cdots,N_{(k,i)},
\end{equation*}
with the same gap $m_{(k,i)}$. Denote by $\mathcal{S}_{(k,i)}$ the set of such shadowing points. It is a $(c_{k,i},3\epsilon^\ast)$-separated set, where $c_{k,i}=n_{(k,i)}N_{(k,i)}+(N_{(k,i)}-1)m_{(k,i)}$.

Let $\mathcal{I}_{(0,0)}=\mathcal{S}_{(0,0)}$ and $M_{(0,0)}=n_{(0,0)}$. Suppose that we have already constructed  the set $\mathcal{I}_{(k,j)^\ast}$. Let $M_{(k,j)}=M_{(k,j)^\ast}+m_{(k,j)}+c_{k,j}$. For $x\in\mathcal{I}_{(k,j)^\ast}$ and $y\in\mathcal{S}_{(k,j)}$, let $z=z(x,y)$ be the shadowing point such that $d_{M_{(k,j)^\ast}}( z,x )<\epsilon_{(k,j)}$ and  $$d_{c_{k,j}}(y,f^{ M_{(k,j)^\ast}+m_{(k,j)} }(z))<\epsilon_{(k,j)}.$$
We set $\mathcal{I}_{(k,j)}=\{ z=z(x,y): x\in\mathcal{I}_{(k,j)^\ast},y\in\mathcal{S}_{(k,j)} \}$, and then define $$\mathbb{F}=\cap_{(k,j)\in\mathcal{A} }\mathbb{F}_{(k,j)},$$ where $\mathbb{F}_{(k,j)}=\cup_{z\in\mathcal{I}_{(k,j)}}\overline{B}_{ M_{(k,j)} }(z,2\epsilon_{(k,j)})$. One can verify that  $\mathbb{F}\subset G_K^{\omega,U}$.

We proceed to define a sequence of measures $\{\rho_{(k,j)}\}_{(k,j)\in\mathcal{A}}$ by
\begin{equation*}
\rho_{(k,j)}:=\frac{1}{\#\mathcal{I}_{(k,j)}}\sum_{x\in\mathcal{I}_{(k,j)} }\delta_x.
\end{equation*}
Without loss of generality, we can assume that  $\rho_{(k,j)}\to\rho$. Furthermore, one can  check that $\rho(\mathbb{F})=1$, and for every sufficiently large $n$ and  $B_{n}(z,\frac{\epsilon^\ast}{2}) \cap \mathbb{F}$, it holds  that
\begin{equation*}
\rho(B_{n}(z,\frac{\epsilon^\ast}{2})  )\leq\exp\left[
-n( 1-\zeta )^2(H-2\zeta)
\right].
\end{equation*} 
Then,  by the Lemma \ref{bowdis} we have
\begin{equation*}
( 1-\zeta )^2(H-2\zeta) \leq h_{top}^B(f,\mathbb{F},\frac{\epsilon^\ast}{2}) \leq h_{top}^B(f,G_K^{\omega,U},\frac{\epsilon^\ast}{2}).
\end{equation*}
Letting $\zeta \to 0$ gives us
$$\inf_{\mu\in K}\inf_{{\rm diam}(\xi)<20\epsilon^\ast}h_{\mu}(f,\xi)\leq h_{top}^B(f,G_K^{\omega,U},\frac{\epsilon^\ast}{2}) $$

Therefore,  by (\ref{3.13}) and (\ref{bl}) we  get
\begin{equation*}
\overline{\rm mdim}_{M}^{B}(f,G_K^{\omega,U},d)=\overline{\rm mdim}_{M}^{B}(f,G_{K},d)=\limsup_{\epsilon\to0}\frac{1}{|\log\epsilon|}\inf_{\mu\in K}\inf_{\mathrm{diam}\xi<\epsilon}h_{\mu}(f,\xi).
\end{equation*}

\subsubsection{Proof of  Theorem \ref{thm 1.3}}

The  more satisfactory  results would be to exchange  the order of  $\limsup_{\epsilon\to 0}$ and $\sup_{\mu \in K}$ (or  $\inf_{\mu \in K}$) of Theorems \ref{thm 1.1} and \ref{thm 1.2}.  However, up to now, the authors remain unaware of whether there exist any counter-examples among infinite entropy systems with the specification property for which the following equalities hold strictly, i.e.,
$$\sup_{\mu \in K}\left\{\limsup_{\epsilon\to0}\frac{1}{|\log\epsilon|}
\inf_{  {\rm diam}(\xi)<\epsilon }h_\mu(f,\xi)\right\}< \overline{\rm mdim}_M^{P}(f,G_K^{\omega,U},d),$$
$$\inf_{\mu \in K}\left\{\limsup_{\epsilon\to0}\frac{1}{|\log\epsilon|}
\inf_{  {\rm diam}(\xi)<\epsilon }h_\mu(f,\xi)\right\}< \overline{\rm mdim}_M^{B}(f,G_K^{\omega,U},d).$$

Inspired  by the definition  of  mean R\'enyi information dimension \cite{Gutman2021}, we define two variants of mean R\'enyi information dimension of invariant measures. 

 Let $K$ be a closed set of $M_f(X)$. For any $\mu \in K$, we use $M_{K}(\mu)$ to denote the space of sequences of invariant measures within $K$ which converge to $\mu$ in the weak$^{*}$-topology. 
 \begin{defn}
 	For any $\mu \in K$, we  define the Bowen  and packing R\'enyi information dimensions of $\mu$ as
 \begin{align*}
 	\underline{\rm mrid}_{\mu}^B(f,K,d):&=\inf_{(\mu_{\epsilon})_{\epsilon} \in M_K(\mu)}\{\liminf_{\epsilon\to 0} \frac{1}{|\log\epsilon|}\inf_{\mathrm{diam}(\xi)<\epsilon}h_{\mu_{\epsilon}}(f,\xi)\},\\
 	\overline{\rm mrid}_{\mu}^P(f,K,d):&=\sup_{(\mu_{\epsilon})_{\epsilon} \in M_K(\mu)}\{\limsup_{\epsilon\to 0} \frac{1}{|\log\epsilon|}\inf_{\mathrm{diam}(\xi)<\epsilon}h_{\mu_{\epsilon}}(f,\xi)\},
 \end{align*}
 respectively.
 \end{defn}
With the help of  Bowen and packing R\'enyi information dimensions, we establish  the variational principles for Bowen and packing metric mean dimensions of  saturated sets, whose form are close to  the ones of Bowen and packing topological entropies of saturated sets \cite{Pfister2007, Zhou2012}.




Now we are ready to prove Theorem \ref{thm 1.3}. 
We first prove the variational principle:
$$\underline{\rm mdim}_{M}^{B}(f,G_{K},d)=\inf_{\mu \in K}\underline{\rm mrid}_{\mu}^B(f,K,d).$$

 Fix $\mu \in K$. For every  $(\mu_{\epsilon})_{\epsilon}\in M_K(\mu)$, by Lemma  \ref{lem 3.6} it holds that $$h_{top}^{B}(f,G_{K},\epsilon)\le\inf_{\mathrm{diam}(\xi)<\epsilon}h_{\mu_{\epsilon}}(f,\xi).$$
This shows that
 $$\underline{\rm mdim}_{M}^{B}(f,G_{K},d)\leq \liminf_{\epsilon\to 0} \frac{1}{|\log\epsilon|}\inf_{\mathrm{diam}(\xi)<\epsilon}h_{\mu_{\epsilon}}(f,\xi)$$ and hence  $	\underline{\rm mdim}_{M}^{B}(f,G_{K},d)\leq \inf_{\mu \in K}\limits \underline{\rm mrid}_{\mu}^B(f,K,d)$.

To show the reverse inequality $$\underline{\rm mdim}_{M}^{B}(f,G_{K},d)\geq \inf_{\mu \in K}\limits \underline{\rm mrid}_{\mu}^B(f,K,d),$$
we  assume that  $\underline{\rm mdim}_{M}^{B}(f,G_{K},d)$ is finite; otherwise there is nothing left to prove.  Let $\gamma >0$. By Theorem \ref{thm 1.2},  we can choose  a strictly  decreasing sequence  $\epsilon_k \to 0$ as $k \to 0$ and a sequence  $\{\mu_k\}$ of measures in $K$ such that for any $k$,
\begin{align}\label{inequ 3.18}
\frac{1}{|\log\epsilon_k|}\inf_{\mathrm{diam}(\xi)<\epsilon_k}h_{\mu_k}(f,\xi)<\underline{\rm mdim}_{M}^{B}(f,G_{K},d)+\gamma.
\end{align}
 Without loss of generality, we assume that $\mu_k \rightarrow \mu$ as $k \to \infty$. Then $\mu \in K$ since $K$ is closed.  For every $\epsilon >0$, if $\epsilon \in (\epsilon_{k+1},\epsilon_k]$ for some $k\geq 1$, we define  $\nu_{\epsilon}:=\epsilon_{k}$. Then $(\nu_\epsilon)\in M_K(\mu)$. Using (\ref{inequ 3.18}), we have  
 $$\underline{\rm mrid}_{\mu}^B(f,K,d)\leq \underline{\rm mdim}_{M}^{B}(f,G_{K},d)+\gamma.$$
 Letting $\gamma \to 0$, we have 
 $$\inf_{\mu \in K}\underline{\rm mrid}_{\mu}^B(f,K,d)\leq \underline{\rm mdim}_{M}^{B}(f,G_{K},d).$$

Next we show  the variational principle:
$$\overline{\rm mdim}_{M}^{P}(f,G_{K},d)= \sup_{\mu \in K}\limits \overline{\rm mrid}_{\mu}^P(f,K,d).$$

 Using  (\ref{3.2}), there exists a constant $c>0$ for every sufficiently small $\epsilon >0$,
$$\sup_{\mu\in K}\inf_{{\rm diam}(\xi)<\epsilon}h_{\mu}(f,\xi)\leq h_{top}^P(f,G_K, c\epsilon).$$
This yields  the inequality $$ \sup_{\mu \in K}\overline{\rm mrid}_{\mu}^P(f,K,d)\leq\overline{\rm mdim}_M^{P}(f,G_K,d).$$
On the other hand, let $\gamma >0$  and put $M_{\gamma}:=\min\{\frac{1}{\gamma},\overline{\rm mdim}_{M}^{P}(f,G_{K},d)-\gamma\}.$ By Theorem \ref{thm 1.1},  we can choose  a sequence $\{\epsilon_k\}_k$ of positive real numbers that converges to $0$ and a sequence  $\{\mu_k\}$ of measures in $K$ such that for sufficiently large $k$,
$$M_{\gamma}< \frac{1}{|\log\epsilon_k|}\inf_{\mathrm{diam}(\xi)<\epsilon_k}h_{\mu_k}(f,\xi).$$
By the similar arguments,  we have 
$$\overline{\rm mdim}_M^{P}(f,G_K,d) \leq \sup_{\mu \in K}\overline{\rm mrid}_{\mu}^P(f,K,d).$$

\subsection{Upper capacity metric mean dimension of  saturated sets}
In this subsection, we  give the proof of Theorem \ref{thm 1.4}.

By Proposition  \ref{prop 2.1}, it is clear that $$\overline{\rm mdim}_M^{UC}(f,G_K^{\omega,U},d)
 \leq \overline{\rm mdim}_M^{UC}(f,G_K,d)
 \leq \overline{\rm mdim}_M(f,X,d)$$
 is clear. Thus, it suffices to show  $$\overline{\rm mdim}_M(f,X,d)\leq \overline{\rm mdim}_M^{UC}(f,G_K^{\omega,U},d).$$ 
 
 In light of  Theorem \ref{mlk}, we need to show that for every sufficiently small $\epsilon^\ast>0$, $\zeta>0$  and $\delta\in(0,1)$, one has
\begin{equation}\label{equ 3.19}
\sup_{\mu\in M_f^e(X)}\underline{h}_{\mu}^K(f,3\epsilon^\ast,\delta)-2\zeta\leq h_{top}^{UC}(f,G_K^{\omega,U},\epsilon^\ast).
\end{equation}
Here,  we  will give a sketch for  the proof of (\ref{equ 3.19}) since it is essentially analogous to that of the lower bound of  $\overline{\rm mdim}^P_M(f,G_K^{\omega,U},d)$ presented in  subsection \ref{lower}.

Write $H=\sup_{\mu\in M_f^e(X)}\underline{h}_{\mu}^K(f,3\epsilon^\ast,\delta)$ and assume that $H>0$. 
Take $\mu^\ast\in M_f^e(X)$ such that
\begin{equation*}
\underline{h}_{\mu^\ast}^K(f,3\epsilon^\ast,\delta)>H-\zeta.
\end{equation*}
Then,  for sufficiently large $\mathcal{N}\in\mathbb{N}$, there exists a $(\mathcal{N},3\epsilon^\ast)$-separated subset $S_\mathcal{N}$ with
\begin{equation*}
\#S_\mathcal{N}\geq \exp[
\mathcal{N}(H-2\zeta)
].
\end{equation*}

Since $K$ is a connected closed subset, a standard argument shows that there exists  a family $\{ \mu_{k,i}: k\geq 1, 1\leq i\leq Q_k\}$ of measures in $K$ such that for all $k\geq 1$, $1\leq i<Q_k$,
\begin{equation*}
K\subset\bigcup_{i=1}^{Q_k}B(\mu_{k,i},\frac{1}{k}),\ D(\mu_{k,i},\mu_{k,i+1})<\frac{1}{k},\ 
\text{and }D(\mu_{k,Q_k},\mu_{k+1,1})<\frac{1}{k}.
\end{equation*} 
Consider the set $\mathcal{A}'=\{  
(k,i): k\geq 1, 1\leq i\leq Q_k
\}$ and  endow $\mathcal{A}=\{(k,0): k\in\mathbb{Z}\ \text{and }k\geq-1 \}\cup\mathcal{A}'$  with the lexicographic order. Similarly, the sequences $\{\epsilon_{(k,j)} \}_{(k,j)\in\mathcal{A}}$, $\{\delta_{(k,j)}\}_{(k,j)\in\mathcal{A}}$, $\{m_{(k,j)} \}_{(k,j)\in\mathcal{A}}$ are constructed  as in subsection \ref{lower}.

For the non-empty open set $U$, there exist $a_0\in X$ and $\epsilon_0\in(0,\epsilon^\ast)$ with ${\overline{B}(a_0,\epsilon_0)}\subset U$.
Set $\mathcal{C}_{(-1,0)}=\{ a_0 \}$, $n_{(-1,0)}=c_{(-1,0)}=N_{(-1,0)}=1$, and  $\mathcal{C}_{(0,0)}=S_{\mathcal{N}}$, $c_{(0,0)}=\mathcal{N}$, $N_{(0,0)}=1$.
For $k\geq 1$, let $\mathcal{C}_{(k,0)}$, $n_{(k,0)}$, $c_{(k,0)}$ and $N_{(k,0)}$ be the same as in subsection \ref{lower}. For $(k,i)\in\mathcal{A}'$, since  $G_{\mu_{k,i}}\neq\emptyset$, one can pick $x_{k,i}\in X$ and $n_{(k,i)}\in\mathbb{N}$ such that
\begin{equation*}
D( \mathcal{E}_{n_{(k,i)}}(x), \mu_{k,i})<\delta_{(k,i)},\ \frac{n_{(k,i)}}{n_{(k,i)}+m_{(k,i)}}>\max\left\{
1-\zeta,1-\frac{1}{k}
\right\}.
\end{equation*}
Set $\mathcal{C}_{(k,i)}=\{ x_{k,i} \}$. Now choose a strictly increasing sequence of integers $\{N_{(i,j)}  \}_{ (k,j)\in\mathcal{A}\ \text{with }j\neq 0 }$ with $N_{(1,1)}=1$ and
\begin{equation*}
\begin{split}
&n_{(k,j)_\ast}+m_{(k,j)_\ast}\leq\frac{1}{k}\sum_{(k',j')\leq(k,j) }\left[
N_{(k',j')}\cdot(n_{(k',j')}+m_{(k',j')} )
\right],\\
&\sum_{(k',j')\leq(k,j) }\left[
N_{(k',j')}\cdot(n_{(k',j')}+m_{(k',j')} )
\right]\leq\frac{ N_{ (k,j)_\ast }\cdot n_{(k,j)_\ast} }{k}.
\end{split}
\end{equation*}
Based on  the data $\mathcal{C}_{(k,j)}$, $N_{(k,j)}$ for $(k,j)\in\mathcal{A}$, similar to the previous subsections, one can construct the intermediate sets $\mathcal{S}_{(k,j)}$, $\mathcal{I}_{(k,j)}$, $\mathbb{F}_{(k,j)}$ to obtain a fractal set
\begin{equation*}
\mathbb{F}=\bigcap_{(k,j)\in\mathcal{A}}\mathbb{F}_{(k,j)}\subset G_K^{\omega,U}.
\end{equation*}
It is worth mentioning that for any  $(k,j)>(0,0)$, $\#\mathcal{I}_{(k,j)}=\# S_{\mathcal{N}}$. Moreover, for any distinct $x$, $x'\in\mathcal{I}_{(0,0)}$, one has
\begin{equation*}
d_{ M_{(0,0)} }(x,x')>2\epsilon^\ast.
\end{equation*}
Recall  $M_{(0,0)}=1+m(\frac{\epsilon^\ast}{8})+\mathcal{N}$.
Let  $y$, $y'\in\mathbb{F}$ with $y\in \overline{B}_{M_{(0,0)} }(x,2\epsilon_{(0,0)})$ and $y'\in \overline{B}_{M_{(0,0)} }(x',2\epsilon_{(0,0)})$. Then
one has
\begin{equation*}
d_{M_{(0,0)} }(y,y')>d_{ M_{(0,0)} }(x,x')-d_{ M_{(0,0)} }(x,y)-d_{ M_{(0,0)} }(x',y')>\epsilon^\ast.
\end{equation*}
Hence, $\mathbb{F}$ is  a $(1+m(\frac{\epsilon^\ast}{8})+\mathcal{N},\epsilon^\ast)$-separated subset with  $\# S_{\mathcal{N}}\geq \exp[\mathcal{N}(H-2\zeta)]$. Thus,
\begin{equation*}
h_{top}^{UC}(f,G_K^{\omega,U},\epsilon^\ast)\geq\limsup_{\mathcal{N}\to\infty}\frac{\mathcal{N}(H-2\zeta)}{1+m(\frac{\epsilon^\ast}{8})+\mathcal{N}}=\sup_{\mu\in M_f^e(X)}\underline{h}_{\mu}^K(f,3\epsilon^\ast,\delta)-2\zeta.
\end{equation*}
This implies  the inequality (\ref{equ 3.19}) by letting $\zeta \to 0$.

\subsection{Metric mean dimensions of sup sets and sub sets}

In this section, we  consider  two  variants of saturated sets, namely  sup sets and sub sets, and establish  variational principles for metric mean dimensions of sup sets and sub sets.

In general, given a point  $x\in X$, the limit $\lim_{n\to\infty}\mathcal{E}_n(x)$ may not exist in the  weak$^{*}$-topology. The set of such points is denoted by  $I(f)$,  known  as  the  historic set\footnote[1]{This terminology is due to Rulle's original paper \cite{rue01}.} of $f$.  To understand the underlying geometry structure of $I(f)$,  we define the following sets based on the framework introduced and further developed by Olsen  \cite{o02,o03,o06,o08} and Olsen and Winter \cite{ow03}. More precisely, given a non-empty subset $C \subset M_f(X)$, we  define the \emph{sup set and sub set of $C$} as
\begin{align*}
^CG:&=\{x\in X:V_f(x)\supset C\},\\
G^C:&=\{x\in X:V_f(x)\subset C\},
\end{align*}
respectively. More generally, for two non-empty subsets \(C_1,C_2 \subset M_f(X)\), we can define
\[
G(C_1, C_2):=\{x\in X:  C_1\subset V_f(x)\subset C_2\},
\]
to  recover the above notions  in an unified manner.

Now, we are ready to establish variational principles for  \(G^{C},^CG\).

\begin{thm}\label{thm 3.7}
Let $(X,d,f)$ be a TDS with the  specification property. Then for any non-empty subset $C\subseteq M_{f}(X)$, 
\begin{align*}
\overline{\rm mdim}_M^{B}\left(f,^CG,d\right)=\limsup_{\epsilon\to0}\frac{1}{|\log\epsilon|}
\inf_{\mu\in C}
\inf_{  {\rm diam}(\xi)<\epsilon }h_\mu(f,\xi).
\end{align*}
\end{thm}

\begin{proof}
We claim that for every $\epsilon >0$,
\begin{equation}\label{iic}
\inf_{\mu\in C}
\inf_{  {\rm diam}(\xi)<\epsilon }h_\mu(f,\xi)=\inf_{\mu\in \overline{{\rm co}
}(C)}\inf_{  {\rm diam}(\xi)<\epsilon }h_\mu(f,\xi),
\end{equation}
where  $\overline{{\rm co}
}(C)$ denotes the  closure of the convex hull of $C$.

Indeed, since $C\subset \overline{{\rm co}
}(C)$, it suffices to show $$\inf_{\mu\in C}
\inf_{  {\rm diam}(\xi)<\epsilon }h_\mu(f,\xi)\leq \inf_{\mu\in \overline{{\rm co}
}(C)}\inf_{  {\rm diam}(\xi)<\epsilon }h_\mu(f,\xi).$$ 
Let $\mu \in  \overline{{\rm co}
}(C)$. Choose $\mu_n=\sum_{j=1}^{k_n}\lambda_n^{(j)}v_n^{(j)}$  with $v_n^{(j)} \in C$, $0<\lambda_n^{(j)}<1$ and  $\sum_{j=0}^{k_n}\lambda_n^{(j)}=1$ such that $\mu_n \rightarrow \mu$ as $n \to \infty$. The upper semi-continuity of $ \inf_{  {\rm diam}(\xi)<\epsilon }h_\mu(f,\xi)$ on $M_f(X)$ (cf. \cite[Lemma 2.3, (3)]{fw16}) implies that
$$\limsup_{n\to \infty}~{\inf_{  {\rm diam}(\xi)<\epsilon }h_{\mu_n}(f,\xi)}\leq \inf_{  {\rm diam}(\xi)<\epsilon }h_\mu(f,\xi).$$
Since  $ \inf_{  {\rm diam}(\xi)<\epsilon }h_{\mu}(f,\xi)$ is  a concave  function on $M_f(X)$,  for every $\gamma >0$  we can deduce that
\begin{equation*}
\inf_{\nu\in C}\inf_{  {\rm diam}(\xi)<\epsilon }h_\nu(f,\xi)< \inf_{  {\rm diam}(\xi)<\epsilon }h_\mu(f,\xi)+\gamma.
\end{equation*}
 This yields the desired   inequality  \eqref{iic}.

Notice that $G_{\overline{{\rm co}}(C)}\subset ^CG$, we have
\begin{align*}
\overline{\rm mdim}_M^{B}\left(f,^CG,d\right)&\geq \overline{\rm mdim}_M^{B}\left(f,G_{\overline{{\rm co}}(C)},d\right)\\
&=\limsup_{\epsilon\to0}\frac{1}{|\log\epsilon|}
\inf_{\mu\in \overline{{\rm co}}(C)}
\inf_{  {\rm diam}(\xi)<\epsilon }h_\mu(f,\xi) ~\text{by Theorem \ref{thm 1.2}}\\
&=\limsup_{\epsilon\to0}\frac{1}{|\log\epsilon|}
\inf_{\mu\in C}
\inf_{  {\rm diam}(\xi)<\epsilon }h_\mu(f,\xi).
\end{align*}

Fix $\mu \in C$ and $\epsilon >0$. By Lemma \ref{lem 3.6}, one has
\begin{align*}
h_{top}^B(f,^CG,\epsilon)\leq  h_{top}^B(f,\{x\in X:\mu \in V_f(x)\},\epsilon)
\leq  \inf_{  {\rm diam}(\xi)<\epsilon }h_\mu(f,\xi),
\end{align*}
and hence $h_{top}^B(f,^CG,\epsilon)\leq  \inf_{\mu\in C}
\inf_{  {\rm diam}(\xi)<\epsilon }h_\mu(f,\xi)$.
\end{proof}

\begin{thm}\label{thm 3.8}
Let $(X,d,f)$ be a TDS with  the specification property.  Then for any non-empty closed subset $C\subseteq M_{f}(X)$, we have
\begin{align*}
\overline{\rm mdim}_M^{B}\left(f,G^C,d\right)=\limsup_{\epsilon\to0}\frac{1}{|\log\epsilon|}
\sup_{\mu\in C}
\inf_{  {\rm diam}(\xi)<\epsilon }h_\mu(f,\xi).
\end{align*}
\end{thm}
\begin{proof}
For every  $\epsilon >0$,  by Lemma \ref{lem 3.6}  we have
\begin{align*}
h_{top}^B(f,G^C,\epsilon)\leq  h_{top}^B(f,\{x\in X: V_f(x)\cap {C}\not=\emptyset\},\epsilon)
\leq  \sup_{\mu \in {C}}\inf_{  {\rm diam}(\xi)<\epsilon }h_\mu(f,\xi).
\end{align*}
By (\ref{bl}), there exists a constant $c>0$ such that  for all $\mu \in C$ and any sufficiently small $\epsilon>0$,
\begin{equation*}
\inf_{{\rm diam}(\xi)<c\epsilon}h_{\mu}(f,\xi)
\leq
h_{top}^B(f,G_{\mu},\epsilon).
\end{equation*}
 Combining this inequality with the fact that $G_\mu\subset G^C$  for each $\mu\in C$, we have
$$\sup_{\mu \in C}\inf_{{\rm diam}(\xi)<c\epsilon}h_{\mu}(f,\xi)
\leq h_{top}^B(f,G^C,\epsilon).$$
\end{proof}

\begin{thm}\label{thm 3.9}
Let $(X,d,f)$ be a TDS with  the specification property. Let $C_1\subseteq M_{f}(X)$  be  a non-empty set and $C_1\subseteq M_{f}(X)$. If $\overline{{\rm co}}(C_1)$ is a connected component of $C_2$, then
\begin{align*}
\overline{\rm mdim}_M^{B}\left(f,G(C_1,C_2),d\right)=\limsup_{\epsilon\to0}\frac{1}{|\log\epsilon|}
\inf_{\mu\in C_1}
\inf_{  {\rm diam}(\xi)<\epsilon }h_\mu(f,\xi).
\end{align*}
\end{thm}

\begin{proof}
Since $C_1\subset\overline{{\rm co}}(C_1)\subset C_2$,   by Theorem \ref{thm 1.2} we have
\begin{align*}
\overline{\rm mdim}_M^{B}\left(f,G(C_1,C_2),d\right)&\geq  \overline{\rm mdim}_M^{B}\left(f,G_{\overline{{\rm co}}(C_1)},d\right)\\
&= \limsup_{\epsilon\to0}\frac{1}{|\log\epsilon|}
\inf_{\mu\in \overline{{\rm co}}(C_1)}
\inf_{  {\rm diam}(\xi)<\epsilon }h_\mu(f,\xi)\\
&= \limsup_{\epsilon\to0}\frac{1}{|\log\epsilon|}
\inf_{\mu\in C_1}
\inf_{  {\rm diam}(\xi)<\epsilon }h_\mu(f,\xi).
\end{align*}
Furthermore, we have $\overline{\rm mdim}_M^{B}\left(f,G(C_1,C_2),d\right)\leq \overline{\rm mdim}_M^{B}\left(f,^{C_1}G,d\right)$.  Using Theorem \ref{thm 3.7},  we get  $$\overline{\rm mdim}_M^{B}\left(f,G(C_1,C_2),d\right) \leq \limsup_{\epsilon\to0}\frac{1}{|\log\epsilon|}
\inf_{\mu\in C_1}
\inf_{  {\rm diam}(\xi)<\epsilon }h_\mu(f,\xi).$$ This completes the proof.
\end{proof}

The Theorems \ref{thm 3.7}, \ref{thm 3.8} and \ref{thm 3.9} imply the following corollaries, which allows us to compare the   metric mean dimensions of sup sets and sub sets with the saturated sets.

\begin{cro}
Let $(X,d,f)$ be a TDS with  the specification property.
\begin{itemize}
\item [(1)] If  $C$ is a non-empty compact connected set of $M_f(X)$, then 
$$\overline{\rm mdim}_M^{B}\left(f,^CG,d\right)=\overline{\rm mdim}_M^{B}\left(f,G_C,d\right).$$
\item [(2)] If  $C$ is a non-empty compact convex set of $M_f(X)$, then 
$$\overline{\rm mdim}_M^{B}\left(f,G^C,d\right)=\overline{\rm mdim}_M^{P}\left(f,G_C,d\right).$$
\end{itemize} 
\end{cro}

\begin{proof}
(1) follows from Theorems \ref{thm 1.2} and \ref{thm 3.7}; (2) is due to Theorems \ref{thm 1.1} and \ref{thm 3.8}.
\end{proof}

\begin{cro}
Let $(X,d,f)$ be a TDS with  the specification property.
\begin{itemize}
\item [(1)] If  $C$ is a non-empty compact connected set of $M_f(X)$, then 
$$\overline{\rm mdim}_M^{B}\left(f,G_C,d\right)=\overline{\rm mdim}_M^{B}\left(f,^CG,d\right)\leq\overline{\rm mdim}_M^{B}\left(f,G^C,d\right).$$
\item [(2)] If  $C$ is a non-empty compact convex set of $M_f(X)$, then 
$$\overline{\rm mdim}_M^{P}\left(f,G_C,d\right)=\overline{\rm mdim}_M^{P}\left(f,G^C,d\right)\leq \overline{\rm mdim}_M^{P}\left(f,^CG,d\right).$$
\item [(3)] If  $C$ is a non-empty compact connected set of $M_f(X)$, then 
$$\overline{\rm mdim}_M^{UC}\left(f,G_C,d\right)=\overline{\rm mdim}_M^{UC}\left(f,G^C,d\right)= \overline{\rm mdim}_M^{UC}\left(f,^CG,d\right).$$
\end{itemize} 
\end{cro}

\begin{proof}
$(1)$ is clear.
 
 $(2)$. It follows from Theorem \ref{thm 1.1} that
\begin{align*}
\limsup_{\epsilon\to0}\frac{1}{|\log\epsilon|}\sup_{\mu\in C}
\inf_{  {\rm diam}(\xi)<\epsilon }h_\mu(f,\xi)=\overline{\rm mdim}_M^{P}\left(f,G_C,d\right)
\leq \overline{\rm mdim}_M^{P}\left(f,G^C,d\right).
\end{align*}
On the other hand, one has $h_{top}^P(f,G^C,\epsilon) \leq \sup_{\mu \in C}\inf_{{\rm diam}(\xi)<\epsilon}h_{\mu}(f,\xi)$ by Lemma \ref{upper}.
Hence, we have  $\overline{\rm mdim}_M^{P}\left(f,G_C,d\right)=\overline{\rm mdim}_M^{P}\left(f,G^C,d\right)$.

 $(3)$. It follows from  Theorem \ref{thm 1.4} and the relations $G_C\subset G^C$ and $G_C\subset ^CG$.
\end{proof}

\section{The applications of main results}\label{sec 4}
In this section, we apply the main results to study the metric mean dimension of level sets, mean Li-Yorke chaos in infinite entropy systems and the metric mean dimension of the set of generic points of  full shifts over compact metric spaces.  

\subsection{Metric mean dimensions of level sets}

Let $(X,d,f)$ be a TDS and $\psi$ be a  continuous function on $X$.  If $\mu$ is a $f$-invariant  probability measure on $X$, the Birkhoff ergodic theorem   verifies that for $\mu$-a.e. $x\in X$, the limit
$$\lim_{n \to \infty}
\frac{1}{n}\sum_{j=0}^{n-1} \psi(f^jx)$$
exists.

However, given $\alpha \in \mathbb{R}$ it does not give more information about the ``size" of the set of the points of $X$ that the time average  of $\psi$ equals $\alpha$. For this reason, the \emph{$\alpha$-level set of $\psi$}  is defined by 
$$K_{\alpha}:=\{x\in X:\lim_{n \to \infty}
\frac{1}{n}\sum_{j=0}^{n-1} \psi(f^jx)=\alpha\}.$$

In \cite{Backes2023}, the authors proved that  if $K_{\alpha}$ is non-empty and the TDS $(X,d,f)$ has the  specification property, then 
\begin{align*}
\overline{\rm mdim}_M^B(f,K_\alpha,d)
=
\limsup_{\epsilon\to0}\frac{ 1 }{|\log\epsilon|}\sup_{\mu\in M_f(\psi,\alpha)}\inf_{ {\rm diam}(\xi)<\epsilon }h_\mu(f,\xi),
\end{align*}
where $M_f(\psi,\alpha):=\{\mu \in M_f(X): \int \psi d\mu =\alpha\}$.

Actually, we are able to prove this variational principle  by  an alternative approach and  give new characterization for the (packing/upper capacity) metric mean dimension of level sets.

\begin{thm}\label{cor 4.2}
	Let $(X,d,f)$ be a TDS with the specification property and $\psi$ be a  continuous function on $X$. 
	If $\alpha \in \mathbb{R}$ is a constant such that $K_{\alpha}\not=\emptyset$, then   
	\begin{align*}
		\overline{\rm mdim}_M^B(f,K_\alpha,d)
		&=
		\overline{\rm mdim}_M^P(f,K_\alpha,d)
		=
		\limsup_{\epsilon\to0}\frac{ 1 }{|\log\epsilon|}\sup_{\mu\in M_f(\psi,\alpha)}\inf_{ {\rm diam}(\xi)<\epsilon }h_\mu(f,\xi),\\
		\overline{\rm mdim}_M^{UC}(f,K_\alpha,d)
		&=\overline{\rm mdim}_M(f,X,d).
	\end{align*}
	The corresponding result are also valid for  the three types of lower metric mean dimensions by changing $\limsup_{\epsilon\to 0}$ into $\liminf_{\epsilon\to 0}$.
\end{thm}

\begin{proof}
	If  $K_{\alpha}\not=\emptyset$, then  $M_f(\psi,\alpha)$ is  a non-empty convex  set.  
For $x \in X$, the statement $\lim_{n \to \infty}
	\frac{1}{n}\sum_{j=0}^{n-1} \psi(f^jx)=\alpha$ is equivalent  to the statement $V_f(x)\subset M_f(\psi, \alpha)$.  Recall $G^{ M_f(\varphi,\alpha) }=\{ x\in X: V_f(x)\subset M_f(\varphi,\alpha)\}$. It follows that $K_{\alpha}=G^{ M_f(\varphi,\alpha) }$.  By Lemma \ref{upper} $(1)$, we get
	\begin{equation}\label{4.3}
		\overline{\rm mdim}_M^P(f,K_\alpha,d)\leq \limsup_{\epsilon\to0}\frac{ 1}{|\log\epsilon|}\sup_{\mu\in M_f(\varphi,\alpha)}\inf_{ {\rm diam}(\xi)<\epsilon }h_\mu(f,\xi).
	\end{equation}
	
By inequality (\ref{bl}), there exists a constant $c>0$ such that  for every $\mu \in M_f(X)$ and sufficiently  small $\epsilon>0$,
\begin{equation*}
	\inf_{{\rm diam}(\xi)<c\epsilon}h_{\mu}(f,\xi)
	\leq
	h_{top}^B(f,G_{\mu},\epsilon).
\end{equation*}
Since $G_\mu\subset K_{\alpha}$  for any $\mu\in M_f(\varphi,\alpha)$,  we have
$$\sup_{\mu \in M_f(\varphi,\alpha)}\inf_{{\rm diam}(\xi)<c\epsilon}h_{\mu}(f,\xi)
\leq
h_{top}^B(f, K_{\alpha},\epsilon),$$
which implies that
\begin{align}\label{4.2}
	\limsup_{\epsilon\to 0}\frac{1}{|\log\epsilon|}\sup_{\mu \in M_f(\varphi,\alpha)}\inf_{{\rm diam}(\xi)<\epsilon}h_{\mu}(f,\xi)
	\leq
	\overline{\rm mdim}_M^B(f,K_\alpha,d).
\end{align}	
By (\ref{4.3}), (\ref{4.2}) and $\overline{\rm mdim}_M^B(f,K_{\alpha},d)\leq \overline{\rm mdim}_M^P(f,K_{\alpha},d)$, we get the first equalities.

Fix  $\mu\in M_f(\psi,\alpha)$.  Since the specification property implies  $G_{\mu}\not=\emptyset$ \cite{Pfister2007},  we have $G_{\mu}\subset K_{\alpha}$. By  Corollary \ref{cor 1.5}, we obtain  $\overline{\rm mdim}_M^{UC}(f,K_\alpha,d)
=\overline{\rm mdim}_M(f,X,d)$.
\end{proof}

\begin{exa}\label{ex 4.2}
	Endow the product space $[0,1]^\mathbb{Z}$ with  the product metric:
	\begin{equation*}
		d( (x_n)_{n\in\mathbb{Z}},(y_n)_{n\in\mathbb{Z}} )
		=
		\sum_{n\in\mathbb{Z}}\frac{|x_n-y_n|}{2^{|n|}},
	\end{equation*}
	and let $\sigma:[0,1]^\mathbb{Z}\to [0,1]^\mathbb{Z}$ be the left shift map defined by $\sigma((x_n)_{n\in\mathbb{Z}})=(x_{n+1})_{n\in\mathbb{Z}}$. 
	
	It is shown that  the system $([0,1]^\mathbb{Z},d,\sigma)$  has the specification property {\rm (cf.  \cite[Proposition 21.2]{Denker1976})}, and ${\rm {mdim}}_M(\sigma,[0,1]^\mathbb{Z},d)=1$ {\rm \cite[E. Example]{Lindenstrauss2018}}. Let $\varphi: [0,1]^\mathbb{Z}\to \mathbb{R}$ be a continuous map given by $\varphi((x_n)_{n\in \mathbb{Z}})=x_0$. Then for any $\alpha \in [0,1]$, by Theorem \ref{cor 4.2} we have
	\begin{align*}
		\overline{\rm mdim}_M^B(\sigma,K_\alpha,d)
		=
		\overline{\rm mdim}_M^P(\sigma,K_\alpha,d)
		=
		\limsup_{\epsilon\to0}\frac{ 1 }{|\log\epsilon|}\sup_{\mu\in M_\sigma(\varphi,\alpha)}\inf_{ {\rm diam}(\xi)<\epsilon }h_\mu(\sigma,\xi),
	\end{align*}
	where  $K_{\alpha}:=\{x\in [0,1]^\mathbb{Z}:\lim_{n \to \infty}
	\frac{1}{n}\sum_{j=0}^{n-1} x_j=\alpha\}$, and 
	$M_{\sigma}(\varphi,\alpha):=\{\mu \in M_{\sigma}([0,1]^\mathbb{Z}): \int \varphi d\mu =\alpha\}$. Furthermore,  we have
	$$	\overline{\rm mdim}_M^{UC}(\sigma,K_\alpha,d)=1.$$
\end{exa}

\subsection{Mean Li-Yorke chaos in infinite entropy systems}

In this subsection, we provide a quantitative result involving the chaotic behavior of infinite entropy systems. Specifically, we prove that the (packing/upper capacity) metric mean dimension of the set of mean Li-Yorke pairs is equal to  twice the metric mean dimension of the phase space.

Given a TDS $(X,d,f)$, this naturally  induces a product system  $(X\times X, d\times d, f\times f)$, where $X\times X$ is a  compact space equipped with the product  metric:
 $$d\times d( (x_1,y_1),(x_2,y_2) )=\max\{d(x_1,x_2),d(y_1,y_2) \},$$ and  $(f\times f)(x,y):=(fx,fy)$.
A pair $(x,y)\in X\times X$ is called  a \emph{mean Li-Yorke pair} if
\begin{equation*}
\liminf_{n\to\infty}\frac{1}{n}\sum_{i=0}^{n-1}d(f^ix,f^iy)=0, ~\text{and}\
\limsup_{n\to\infty}\frac{1}{n}\sum_{i=0}^{n-1}d(f^ix,f^iy)>0.
\end{equation*}
Namely, the orbits of $x$ and $y$  is sufficiently close to each other for sufficiently $n$, but they may  be apart from each other for infinitely many time.

Denote by  ${\rm MLY}(f)$ the set of   all mean Li-Yorke pairs.  A subset $S \subset X$ is called  a \emph{mean Li-Yorke set} if any two distinct points in $S$ form a mean Li-Yorke pair. The system $(X,f)$ is \emph{mean Li-Yorke chaotic} if there is an uncountable mean Li-Yorke set.  A well-known fact  is  that positive entropy systems are  mean Li-Yorke chaotic \cite{Dow2014}.

\begin{thm}\label{thm 4.2}
Let $(X,d,f)$ be a  TDS with  the specification property.  Then for $S\in\{UC, P \}$,  we have
\begin{align*}
 &\overline{\rm mdim}_M^{S}(f\times f, {\rm MLY}(f),d\times d)
=2~\overline{\rm mdim}_M(f,X,d)\\
&\underline{\rm mdim}_M^{S}(f\times f, {\rm MLY}(f),d\times d)
=2~\underline{\rm mdim}_M(f,X,d).
\end{align*}
\end{thm}

\begin{proof}
 It  suffices to prove the first equality since the second one can be  obtained in a similar manner. 
 
 We claim that
 there exists $\mu_0\in M_{f\times f}(X\times X)$  such that  $$\int_{X\times X} d(x,y)\dif\mu_0(x,y)>0.$$

Notice that the specification property implies that the system has positive topology entropy. Since positive entropy systems are mean Li-Yorke chaotic \cite{Dow2014}, we  choose two distinct points  $x,y \in X$,  a sequence  $\{n_i\}$ of positive integers that converges to $\infty$  as $i \to \infty$,  and $\theta>0$ (depending on $x,y$) such that   $\frac{1}{n_i}\sum_{j=0}^{n_i-1}d(f^jx,f^jy)>\theta.$  
Put  $\mu_{n_i}:=\frac{1}{n_i}\sum_{j=0}^{n_i-1}\delta_{(f^jx,f^jy)}$.  Then, in the weak$^{*}$-topology, any  accumulation point $\mu_0 \in M_{f\times f}(X\times X)$ of $\{\mu_{n_i}\}$ satisfies the desired  claim.

Fix $\nu\in M_f(X)$ and put $\mu_1:=\nu\circ\phi^{-1}$, where $\phi:X\to X\times X$ is the diagonal map given by  $\phi(x)=(x,x)$. It is easy to see that $\mu_1\in M_{f\times f}(X\times X)$ and $\int_{X\times X} d(x,y)\dif\mu_1=0$.  Fix  $\gamma>0$ and  let $$c_\gamma:=\min\{\frac{1}{\gamma},\overline{\rm mdim}_M(f\times f,X\times X,d\times d)-\gamma\}.$$  By  Theorem \ref{mlk},  we choose  a sequence of positive real numbers $\epsilon_k\to0$ as $k \to \infty$ and a sequence $\{m_k\}$ of $f\times f$-invariant measures such that for every $k$,
\begin{align*}
\frac{ 1}{|\log\epsilon_k|}\inf_{ {\rm diam}(\xi,d\times d)<\epsilon_k }h_{m_k}(f\times f,\xi)>c_{\gamma}.
\end{align*}

Let   $$K=\overline{ {\rm co}}\{\mu_0, \mu_1, m_k, k\geq 1\}.$$ Then $K$ is a closed convex subset of $M_{f\times f}(X\times X)$.  Since $(X\times X, d\times dm f\times f)$ has  the specification property,  by \cite{Pfister2007} we have $G_K\not=\emptyset$. For any $(x_0,y_0)\in G_K$, there exists two sequences of integers $\{ t_n^0 \}$ and $\{ t_n^1 \}$ such that $\mathcal{E}_{ t_n^0 }(x_0,y_0)\to\mu_0$, $\mathcal{E}_{ t_n^1 }(x_0,y_0)\to\mu_1$ as $n\to\infty$, respectively. Hence,
\begin{equation*}
\lim_{n\to\infty}\frac{1}{t_n^0}\sum_{i=0}^{t_n^0-1}d(f^ix_0,f^iy_0)=\int_{X\times X} d(x,y)\dif\mu_0>0,
\end{equation*}
and 
\begin{equation*}
\lim_{n\to\infty}\frac{1}{t_n^1}\sum_{i=0}^{t_n^1-1}d(f^ix_0,f^iy_0)=\int_{X\times X} d(x,y)\dif\mu_1=0.
\end{equation*}
This yields that  $G_K\subset {\rm MLY}(f)$.  
  Therefore, by Theorem \ref{thm 1.1} one  has
\begin{align*}
\overline{\rm mdim}_M^P(f\times f, {\rm MLY}(f),d\times d)
\geq&
\overline{\rm mdim}_M^P(f\times f, G_K,d\times d)\\
=&
\limsup_{\epsilon\to 0}\frac{ 1}{|\log\epsilon|}\sup_{\nu\in K}\inf_{ {\rm diam}(\eta,d\times d)<\epsilon }h_\nu(f\times f,\eta) \\
\geq&
\limsup_{k\to\infty}\frac{ 1}{|\log\epsilon_k|}\inf_{ {\rm diam}(\xi,d\times d)<\epsilon_k }h_{m_k}(f\times f,\xi) \geq c_{\gamma}.
\end{align*}
Letting $\gamma \to 0$, one has
\begin{align*}
 \overline{\rm mdim}_M(f\times f,X\times X,d\times d) \leq& \overline{\rm mdim}^{P}_M(f\times f,{\rm MLY}(f),d\times d)\\
 \leq & \overline{\rm mdim}^{UC}_M(f\times f,{\rm MLY}(f),d\times d)\\
\leq &\overline{\rm mdim}_M(f\times f,X\times X,d\times d)\\
=&2~\overline{\rm mdim}_M(f,X,d).
\end{align*}
This completes the proof.
\end{proof}

For instance, under the setting of  Example \ref{ex 4.2} we have
$$\overline{\rm mdim}_M^{P}(\sigma\times \sigma, {\rm MLY}(\sigma),d\times d)=\overline{\rm mdim}_M^{UC}(\sigma\times \sigma, {\rm MLY}(\sigma),d\times d)=2.$$

\subsection{Generic points  of full shift over compact metric space}\label{sec 5}

For  the full shifts over  compact metric spaces,   we   present the precise formulae for the metric mean dimension of the set of generic points of  invariant measures.

Let $(K,d_K)$ be a compact metric space. Consider the left shift map $\sigma$ on $K^\mathbb{Z}$ again, where  $K^\mathbb{Z}$ is endowed with  the following product metric
\begin{equation*}
d\left(
(x_n)_{n\in\mathbb{Z}}, (y_n)_{n\in\mathbb{Z}}
\right)
=\sum_{n\in\mathbb{Z}}\frac{d_K(x_n,y_n)}{2^{|n|}}.
\end{equation*}
 It is well-known that the system $(K^{\mathbb{Z}},d,\sigma)$ has the specification property. It was proved in \rm{\cite[Theorem 5]{vv7}} that,
\begin{align}\label{md}
\overline{\rm mdim}_M(\sigma,K^{\mathbb{Z}},d)&= \overline{\rm dim}_B(K,d_K),\\
\underline{\rm mdim}_M(\sigma,K^{\mathbb{Z}},d)&= \underline{\rm dim}_B(K,d_K),
\end{align}
where $\overline{\rm dim}_B(K,d)$ and $\underline{\rm dim}_B(K,d)$ denote the upper  and lower box dimensions of $K$ under the metric $d$, respectively. 

By Corollary \ref{cor 1.5}, for  the full shift over compact metric spaces  the metric mean dimensions of the set of generic points of  invariant measures are fully determined  by  the box dimension of the underlying space and mean  R\'enyi information dimension of invariant measures.

 \begin{thm}
Let $(K^{\mathbb{Z}},d,\sigma)$  be the full shift over $(K,d_K)$. 	Then for any  $\mu\in M_\sigma(K^\mathbb{Z})$,
 	\begin{align*}
 		&\overline{\rm mdim}_M^{UC}(\sigma,G_\mu,d^\mathbb{Z})=\overline{\rm dim}_B(K,d)\\
 		&\overline{\rm mdim}_M^B(\sigma,G_\mu,d^\mathbb{Z})=\overline{\rm mdim}_M^P(\sigma,G_\mu,d^\mathbb{Z})
 		=
 		\limsup_{\epsilon\to0}\frac{1}{|\log\epsilon|}\inf_{{\rm diam}(\xi)<\epsilon} h_\mu(\sigma,\xi).
 	\end{align*}
 \end{thm}

Next, for certain   full shift over self-similar set we show the metric mean dimension of the product of the self-similar measure is  closely related to the similarity dimension of IFS.

Let $m\geq 1$ be a positive integer and endow $\mathbb{R}^{m}$  with the  supremum norm $\|\cdot\|_{\mathbb{R}^m}$.  Assume that $T_i:\mathbb{R}^m\to\mathbb{R}^m$  is contracting similitude for $1\leq i\leq l$, i.e., for some $r_i\in(0,1)$, $$\|T_ix-T_iy\|_{\mathbb{R}^m}=r_i\|x-y\|_{\mathbb{R}^m}.$$

The iterated function system (IFS for short) $\{T_i\}_{1\leq i\leq l}$ satisfies the \emph{Open Set Condition} if there exists a bounded open subset $V\subset\mathbb{R}^m$ such that $T_i(V)\subset V,\ \forall 1\leq i\leq l$, and  $T_i(V)\cap T_j(V)=\emptyset ~\text{for}~i\neq j.$ For instance, the middle-third Cantor set, the  Sierpin\'{n}ski triangle or  gasket, are the  attractors of iterated function systems satisfying the Open Set Condition. Under the OSC condition, it is well known that there are  (unique) compact subset $K$, which we call  the \emph{self-similar set}, and  a measure $\mu$, which  we call the \emph{self-similar measure}, such that
\begin{itemize}
	\item [(1)] $K=\cup_{i=1}^lT_i(K)$;
	\item [(2)] ${\rm dim}_H(K,d)={\rm dim}_B(K,d)=\alpha$, and $\alpha$ is the unique solution of equation
	\begin{equation*}
		\sum_{i=1}^lr_i^\alpha=1,
	\end{equation*}
where ${\rm dim}_H$ denotes the Hausdorff dimension under the metric  $d:=\|\cdot\|_{\mathbb{R}^m}$. The root $\alpha$ is called the \emph{similarity dimension} of the IFS;
	\item[(3)]  $\mu=\sum_{i=1}^lr_i^\alpha\cdot \mu\circ T_i^{-1}$ and ${\rm supp}(\mu)=K$.
\end{itemize}

Now  take the attractor  $K$ obtained  as above   and  the product measure $\nu:=\mu^{\otimes\mathbb{Z}}$  on   $K^\mathbb{Z}$.

\begin{thm}
The three types of metric mean dimensions of $G_\nu$ coincide with the  similarity dimension. i.e.,
$$\overline{\rm mdim}_M^S(\sigma,G_{\nu},d^\mathbb{Z})=\alpha=\limsup_{\epsilon\to0}
\frac{\inf_{{\rm diam}(\xi)<\epsilon}h_{\nu}(\sigma,\xi) }{|\log\epsilon|},$$  
where $S\in\{ B, P, UC \}$.
\end{thm}

\begin{proof}
In \cite{Hutchinson1981}, Hutchinson  proved that  there exists a constant $C>0$ such that for any $z\in K$ and $\epsilon\in(0,1)$, one has
\begin{equation}\label{dan}
\mu(B(z,\epsilon))\leq C\epsilon^\alpha.
\end{equation}
 Fix $\epsilon\in(0,1)$. For any $x=(x_k)_{k\in\mathbb{Z}}\in K^\mathbb{Z}$, we  define
\begin{equation*}
L_n(x,\epsilon)=\left\{
(y_k)_{k\in\mathbb{Z}}\in K^\mathbb{Z}: \|x_k-y_k\|_{\mathbb{R}^m}<\epsilon,\ 0\leq k\leq n-1
\right\}.
\end{equation*}
Then  $B_n(x,\epsilon) \subset L_n(x,\epsilon)$. Since $\nu(L_n(x,\epsilon))\leq C^n\epsilon^{n\alpha}$ by  \eqref{dan}, one has
\begin{equation*}
\nu(B_n(x,\epsilon))\leq C^n\epsilon^{n\alpha},\ \forall x\in K^\mathbb{Z},
\end{equation*}
which yields that the  Brin-Katok local $\epsilon$-entropy of $\nu$ is bounded blew by  $$\overline{h}_{\nu}^{BK}(\sigma,\epsilon):=\int \limsup_{n \to \infty}\frac{-\log\mu(B_n(x,\epsilon))}{n} d\nu \geq \alpha|\log\epsilon|-\log C.$$
Applying the Shannon-McMillan-Breiman theorem, the following inequality is valid  for any TDS $(X,d,f)$:
\begin{equation}\label{bkd}
\overline{h}_\mu^{BK}(f,\epsilon)\leq\inf_{{\rm diam}(\xi)<\epsilon}h_\mu(f,\xi),\ \forall\mu\in M_f(X).
\end{equation}
Hence, we have
\begin{align*}
\alpha
&\leq
\limsup_{\epsilon\to0}\frac{\overline{h}_{\nu}^{BK}(\sigma,\epsilon) }{|\log\epsilon|}
\overset{\eqref{bkd}}{\leq}\limsup_{\epsilon\to0}
\frac{\inf_{{\rm diam}(\xi)<\epsilon}h_{\nu}(\sigma,\xi) }{|\log\epsilon|}\\
&\overset{by~Corollary~\ref{cor 1.5}}{\leq}
\overline{\rm mdim}_M^S(\sigma,G_{\nu},d^\mathbb{Z})\leq 
\overline{\rm mdim}_M(\sigma,K^\mathbb{Z},d^\mathbb{Z})\overset{\eqref{md}}{=}{\rm dim}_B(K,d)=\alpha.
\end{align*}
This implies the desired result.
\end{proof}

\noindent
\textbf{Acknowledgments:} The first and second authors were supported by the National Natural Science Foundation of China (No. 12101340).  The third author was  supported by  the China Postdoctoral Science Foundation (No. 2024M763856)  and   the Postdoctoral Fellowship Program of CPSF  (No. GZC20252040); Rui Yang would like to thank Dr. Yunping Wang for many helpful discussions and the hospitality extended to him during his visit to Ningbo University.
Special thanks also go to Prof. Ercai Chen for his many useful comments.

\end{document}